\pdfoutput=1
\RequirePackage{ifpdf}
\ifpdf 
\documentclass[pdftex]{sigma}
\else
\documentclass{sigma}
\fi

\numberwithin{equation}{section}

\newtheorem{Theorem}{Theorem}[section]
\newtheorem{Lemma}[Theorem]{Lemma}
\newtheorem{Proposition}[Theorem]{Proposition}
{\theoremstyle{definition}
\newtheorem{Remark}[Theorem]{Remark}
}

\begin{document}
\allowdisplaybreaks

\renewcommand{\thefootnote}{$\star$}

\newcommand{\arXivNumber}{1504.00715}

\renewcommand{\PaperNumber}{025}

\FirstPageHeading

\ShortArticleName{Loops in SU(2), Riemann Surfaces, and Factorization,~I}

\ArticleName{Loops in SU(2), Riemann Surfaces,\\ and Factorization,~I\footnote{This paper is a~contribution to the Special Issue
on Asymptotics and Universality in Random Matrices, Random Growth Processes, Integrable Systems and Statistical Physics in honor of Percy Deift and Craig Tracy.
The full collection is available at \href{http://www.emis.de/journals/SIGMA/Deift-Tracy.html}{http://www.emis.de/journals/SIGMA/Deift-Tracy.html}}}

\Author{Estelle BASOR~$^\dag$ and Doug PICKRELL~$^\ddag$}
\Address{$^\dag$~American Institute of Mathematics, 600 E.~Brokaw Road, San Jose, CA 95112, USA}
\EmailD{\href{mailto:ebasor@aimath.org}{ebasor@aimath.org}}

\Address{$^\ddag$~Mathematics Department, University of Arizona, Tucson, AZ 85721, USA}
\EmailD{\href{mailto:pickrell@math.arizona.edu}{pickrell@math.arizona.edu}}
\AuthorNameForHeading{E.~Basor and D.~Pickrell}

\ArticleDates{Received October 24, 2015, in f\/inal form March 02, 2016; Published online March 08, 2016}

\Abstract{In previous work we showed that a loop
$g\colon S^1 \to {\rm SU}(2)$ has a triangular factorization if and only if
the loop $g$ has a~root subgroup factorization. In this paper we present generalizations in which
the unit disk and its double, the sphere, are replaced by a~based
compact Riemann surface with boundary, and its double. One ingredient
is the theory of generalized Fourier--Laurent expansions developed by Krichever and Novikov.
We show that a ${\rm SU}(2)$ valued multiloop having an analogue of a root subgroup factorization
sa\-tisf\/ies the condition that the multiloop, viewed as a~transition function, def\/ines
a~semistable holomorphic ${\rm SL}(2,\mathbb C)$ bundle. Additionally, for such a multiloop,
there is a cor\-respon\-ding factorization for
determinants associated to the spin Toeplitz operators def\/ined by the multiloop.}


\Classification{22E67; 47A68; 47B35}

\renewcommand{\thefootnote}{\arabic{footnote}}
\setcounter{footnote}{0}


\bigskip

This paper is dedicated to Percy Deift and Craig Tracy, both of whom have
contributed to the richness and beauty of mathematics and mathematical physics in
so many ways. Both have created techniques and theories that have enhanced
the understanding of fundamental problems and questions. And both
with their enthusiasm and fearlessness, have encouraged, inspired, and motivated others
to try to do the same.

The authors of this paper met in a summer meeting in 1984 in Laramie, Wyoming, where
Craig was one of the principal speakers. One of the topics discussed informally at
the \mbox{meeting} was determinants of Toeplitz operators, a subject where both Percy
and Craig have made considerable contributions, in particular with the study
of singular symbols and applications to statistical mechanics. This paper
concerns some generalized Toeplitz determinant calculations
for matrix-valued symbols, a topic that is a return to a long ago summer,
yet one which is still of current interest.

\section{Introduction}\label{Introduction}

Suppose that $\Sigma$ is a connected compact Riemann surface with nonempty
boundary $S$ (a disjoint union of circles). Let $\widehat{\Sigma}$
denote the double,
\begin{gather*}\widehat{\Sigma}=\Sigma^* \circ \Sigma,\end{gather*}
where $\Sigma^*$ is the adjoint of $\Sigma$, i.e., the surface $\Sigma$
with the orientation reversed, and the composition is sewing along
the common boundary~$S$. Let~$R$ denote the antiholomorphic
involution (or ref\/lection) f\/ixing~$S$.

The classical example is $\Sigma=D$, the closed unit disk. In this case
$S=S^1$, $\widehat{\Sigma}$ is isomorphic to the Riemann sphere, and (in this
realization)
$R(z)=1/z^*$, where $z^*=\overline{z}$, the complex conjugate.
This example has the exceptional feature that there is a
large automorphism group, ${\rm PSU}(1,1)$, acting by linear
fractional transformations.

We now choose a basepoint, denoted by $(0)$, in the
interior of $\Sigma$, and we let $(\infty)$ denote the ref\/lected basepoint for
$\Sigma^*$. In the classical case, without loss of generality because of the ${\rm PSU}(1,1)$ symmetry,
we can assume the basepoint is $z=0$. Given the data $(\Sigma,(0))$, following ideas
of Krichever and Novikov, a reasonable function on $S$ has a `linear triangular factorization'
\begin{gather}\label{LinearTF}
f=f_-+f_0+f_+,
\end{gather}
where $f_{\pm}$ is holomorphic in the interior of $\Sigma$ ($\Sigma^*$, respectively), with appropriate boundary
behavior, depending on the smoothness of $f$, $f_{+}((0))=0$, $f_{-}((\infty))=0$, and $f_0$ is the restriction
to $S$ of a meromorphic function which
belongs to a $\text{genus}(\widehat{\Sigma})+1$-dimensional complementary subspace, which we refer to
as the vector space of zero modes (see Proposition~\ref{hyperfunction}).
In the classical case~$f_0$ is the zero mode for the Fourier series of~$f$.

A holomorphic map $\mathfrak z\colon \widehat{\Sigma}\to \widehat{D}$ is said to be strictly equivariant
if it satisf\/ies
\begin{gather*}
\mathfrak z(R(q))=\frac1{\mathfrak z(q)^*}
\end{gather*}
and maps $\Sigma$ to $D$ (and hence
$\Sigma^*$ to $D^*$). When we refer to the classical case ($\Sigma=D$),
it will be understood that $\mathfrak z(z)=z$. For a function $f\colon U\subset \widehat{\Sigma}\to \mathcal L\big(\mathbb C^N\big)$, def\/ine
$f^*(q)=f(R(q))^*$, where $(\cdot)^*$ is the Hermitian adjoint. If $f\in H^0(\Sigma)$ (i.e., a holomorphic function in
some open neighborhood of $\Sigma$), then $f^*\in H^0(\Sigma^*)$. If $q\in S$, then $f^*(q)=f(q)^*$, the ordinary
complex conjugate of $f(q)$.

\begin{Theorem}\label{SU(2)theorem1} Suppose that $k_1 \in C^{\infty}(S,{\rm SU}(2))$. Consider the following
three conditions:{\samepage
\begin{enumerate}\itemsep=0pt
\item[{\rm (I.1)}] $k_1$ is of the form
\begin{gather*}k_1(z)=\left(\begin{matrix} a(z)&b(z)\\
-b^*(z)&a^*(z)\end{matrix} \right),\qquad z\in S,\end{gather*} where $a$ and
$b$ are boundary values of holomorphic function in $\Sigma$ with $a((0))>0$,
and $a$ and $b$
do not simultaneously vanish at a point in~$\Sigma$.

\item[{\rm (I.2)}] $k_1$ has a ``root subgroup factorization'' of the form
\begin{gather*}k_1(z)=\lim_{n\to\infty}\mathbf a(\eta_n)\left(\begin{matrix} 1&-\overline{\eta}_n\mathfrak z^n\\
\eta_n\mathfrak z^{-n}&1\end{matrix} \right)\cdots \mathbf
a(\eta_0)\left(\begin{matrix} 1&
-\overline{\eta}_0\\
\eta_0&1\end{matrix} \right),\end{gather*} for some rapidly decreasing sequence
$\{\eta_0,\dots,\eta_n,\dots\}$ of complex numbers, and for some strictly equivariant
function~$\mathfrak z$ with~$\mathfrak z((0))=0$.

\item[{\rm (I.3)}] $k_1$ has a (multiplicative triangular) factorization of the form
\begin{gather*}\left(\begin{matrix} 1&0\\ y^*(z)+y_0(z)&1\end{matrix} \right)\left(\begin{matrix} a_1&0\\
0&a_1^{-1}\end{matrix} \right)\left(\begin{matrix} \alpha_1 (z)&\beta_1 (z)\\
\gamma_1 (z)&\delta_1 (z)\end{matrix} \right),\end{gather*} where $a_1>0$,
the third factor is a ${\rm SL}(2,\mathbb C)$-valued holomorphic function in $\Sigma$ which is unipotent upper
triangular at $0$, $y=y_+$ is holomorphic in $\Sigma$, and $y_0$ is a zero mode $($as in~\eqref{LinearTF}$)$.
\end{enumerate}
Then {\rm (I.2)} implies {\rm (I.1)} and {\rm (I.3)} $($with $y_0=0)$,
and {\rm (I.1)} and {\rm (I.3)} are equivalent.}

Similarly, consider $k_2 \in C^{\infty}(S,{\rm SU}(2))$ and the following statements:{\samepage
\begin{enumerate}\itemsep=0pt
\item[{\rm (II.1)}] $k_2$ is of the form
\begin{gather*}k_2(z)=\left(\begin{matrix} d^{*}(z)&-c^{*}(z)\\
c(z)&d(z)\end{matrix} \right),\qquad z\in S^1,\end{gather*} where $c$ and $d$
are boundary values of holomorphic functions in $\Sigma$, $c((0))=0$, and \mbox{$d((0))>0$},
and $c$ and $d$
do not simultaneously vanish at a point in $\Sigma$.

\item[{\rm (II.2)}] $k_2$ has a ``root subgroup factorization'' of the form
\begin{gather*}k_2(z)=\lim_{n\to\infty}\mathbf a(\zeta_n)\left(\begin{matrix} 1&\zeta_n\mathfrak z^{-n}\\
-\overline{\zeta}_n\mathfrak z^n&1\end{matrix} \right)\cdots \mathbf
a(\zeta_1)\left(\begin{matrix} 1&
\zeta_1\mathfrak z^{-1}\\
-\overline{\zeta}_1\mathfrak z&1\end{matrix} \right),\end{gather*} for some rapidly decreasing sequence
$\{\zeta_1,\dots,\zeta_n,\dots\}$ of complex numbers, and for some strictly equivariant~$\mathfrak z$ with
$\mathfrak z((0))=0$.

\item[{\rm (II.3)}] $k_2$ has $($triangular$)$ factorization of the form
\begin{gather*}\left(\begin{matrix} 1& x^*(z)+x_0(z)\\
0&1\end{matrix} \right)\left(\begin{matrix} a_2&0\\
0&a_2^{-1}\end{matrix} \right)\left(\begin{matrix} \alpha_2 (z)&\beta_2 (z)\\
\gamma_2 (z)&\delta_2 (z)\end{matrix} \right), \end{gather*} where $a_2>0$, the third factor is a
${\rm SL}(2,\mathbb C)$-valued holomorphic function which is unipotent
upper triangular at $(0)$, $x=x_+$ is holomorphic
in $\Sigma$, and $x_0$ is a zero mode $($as in \eqref{LinearTF}$)$.
\end{enumerate}
Then {\rm (II.2)} implies {\rm (II.1)} and {\rm (II.3)} $($with $x_0=0)$, and {\rm (II.1)} and {\rm (II.3)} are equivalent.}
\end{Theorem}

\begin{Remark}\label{thm1remarks} \quad
\begin{enumerate}\itemsep=0pt
\item[(a)] In the classical case, all three conditions are equivalent (with $\mathfrak z=z$),
and there are no constraints on $y^*$ or $x^*$ in (I.3) and~(II.3), respectively (see~\cite{PJFA}).

\item[(b)] In nonclassical cases there does not exist an equivariant function $\mathfrak z$ for which
all three conditions are equivalent (because for example~(II.2) implies $x_0=0$ in~(II.3)).
In this paper we will primarily work with the conditions
(I.1) and (I.3), and (II.1) and~(II.3).

\item[(c)] The second condition shows how to generate examples of loops with $y_0=0$ and $x_0=0$ in
(I.3) and (II.3), respectively. We lack a method to generate transparent examples with $y_0\ne 0$
or $x_0\ne 0$. At this point we do not know whether there are constraints on the set of $x^*+x_0$ which
occur in~(II.3).

\item[(d)] Below we will explain why it is essential to generalize this theorem by allowing~$a$ and~$b$ ($c$~and~$d$) to
simultaneously vanish in~$\Sigma$ (in a controlled way) in~(I.1)~((II.1), respectively), and to allow~$a_1$~($a_2$) to be a function in~(I.3) ((II.3), respectively). A tentative step in this direction is discussed in Section~\ref{thm1refinement}.
\end{enumerate}
\end{Remark}

In the classical case a loop $g\colon S^1\to {\rm SU}(2)$ has a root subgroup factorization
\begin{gather*}
g=k_1^* \operatorname{diag}\big(e^{\chi},e^{-\chi}\big) k_2
\end{gather*}
if and only if $g$ has a triangular factorization, a generic condition, and these factorizations are unique (see~\cite{PJFA}). In our more general context we ask similar questions:
can we characterize the set of $g$ which have a
factorization $g=k_1^*\operatorname{diag}(e^{\chi},e^{-\chi}) k_2$, is the relevant condition generic, and is the factorization unique?
An immediate obstacle is that, to our knowledge, there does not currently exist an analogue of the theory of triangular (or Birkhof\/f or Riemann--Hilbert) factorization for matrix valued multiloops in the context of this paper (see~\cite{Rodin} for a survey of results in this direction, and~\cite{CG} or~\cite[Chapter~8]{PS} for further background and history). Partly for this reason, we express what we know about these questions in terms of the holomorphic
bundle def\/ined by viewing a loop as a transition function. For example in the classical case, as observed by Grothendieck, a~loop $g\colon S^1 \to {\rm SL}(2,\mathbb C)$ has a Riemann--Hilbert factorization $g=g_- g_0 g_+$, where $g_-\in H^0(\Delta^*,{\rm SL}(2,\mathbb C))$, $g_0\in {\rm SL}(2,\mathbb C)$, and $g_+\in H^0(\Delta,{\rm SL}(2,\mathbb C))$, if and only if the corresponding holomorphic ${\rm SL}(2,\mathbb C)$ bundle on $\mathbb P^1$ is trivial (see~\cite{Grothendieck}, or~\cite[Chapter~8]{PS}); in the case of the sphere, such a bundle is trivial if and only if it is semistable (for the def\/inition and basic properties of semistability of principal bundles, see~\cite{Ramanathan}; for an alternative exposition, which inspires some of the goals of our project, see~\cite{Bott}).

\begin{Theorem}\label{SU(2)theorem2} Suppose $g\in C^{\infty}(S,{\rm SU}(2))$. If $g$ has a factorization
\begin{gather*}g(z)=k_1^*(z)\left(\begin{matrix} e^{\chi(z)}&0\\
0&e^{-\chi(z)}\end{matrix}\right)k_2(z),\end{gather*} where $\chi \in
C^{\infty}(S,i\mathbb R)$, and $k_1$ and $k_2$ are as in {\rm (I.1)} and {\rm (II.1)}, respectively, of Theorem~{\rm \ref{SU(2)theorem1}},
then~$E(g)$, the holomorphic ${\rm SL}(2,\mathbb C)$ bundle on~$\widehat{\Sigma}$
defined by~$g$ as a transition function, is semistable, and the associated bundle for the defining representation
has a sub-line bundle with an antiholomorphic reflection symmetry compatible with~$R$.
\end{Theorem}

\begin{Remark}\label{thm2remarks}\quad
\begin{enumerate}
\item[(a)] For a multiloop $g$ as in the theorem (or more generally for a multiloop with minimal
smoothness), we explain how $g$ def\/ines a holomorphic bundle in Appendix~\ref{hyperfunctions}.
\item[(b)] Because of the residual symmetry condition on the sub-line bundle in the theorem, the existence
of a factorization as in the theorem is not a generic condition, except in the classical case (this is explained further
in Appendix~\ref{intuition}).

\item[(c)] As an example of our ignorance, suppose that $g\colon S\to {\rm SU}(2)$ has a factorization as in the theorem. Is the same true for
the inverse, $g^{*}$?
This is true in the classical case, because in the classical case there exists a root subgroup factorization if and only if
there is a triangular factorization, and for the latter condition, $g=lmau$ is a triangular factorization if
and only if $g^{-1}=u^*m^*al^*$ is a triangular factorization. We suspect the answer is negative in general.
\end{enumerate}
\end{Remark}

A basic open question is whether a generic $g\colon S\to {\rm SU}(2)$ can be factored as a product of ${\rm SU}(2)$-valued multiloops
\begin{gather}\label{factorize}g(z)=\left(\begin{matrix} a^*(z)&-b(z)\\
b^*(z)&a(z)\end{matrix} \right)\left(\begin{matrix}e^{\chi(z)}&0\\0&e^{-\chi(z)}\end{matrix}\right)\left(\begin{matrix} d^{*}(z)&-c^{*}(z)\\
c(z)&d(z)\end{matrix} \right),
\end{gather} where $a$, $b$, $c$, $d$ are appropriate boundary values (depending on the smoothness of $g$)
of holomorphic functions in~$\Sigma$. Note that in Theorem~\ref{SU(2)theorem2}, $a$ and~$b$ ($c$ and $d$) are not allowed to simultaneously
vanish in $\Sigma$; the point is that this condition has to be relaxed. We will discuss some results along these lines in Sections~\ref{thm1refinement}, \ref{theorem2generalization} and~\ref{generic}.

\subsection{Spin Toeplitz operators}

Assume that $\Sigma$ has a spin structure. There is an induced spin structure for $\widehat{\Sigma}$ which has an
antiholomorphic ref\/lection symmetry compatible with~$R$ (see~\cite[Chapter~7]{Segal}). For simplicity we additionally
assume that the $\overline \partial$
operator for spinors on $\widehat \Sigma$ is invertible. In this case there is a~(pre-)Hilbert space polarization
for the space of ($\mathbb C^2$ valued) spinors along~$S$,
\begin{gather*}
\Omega^{1/2}(S)\otimes\mathbb C^2=H^{1/2}(\Sigma)\otimes\mathbb C^2\oplus H^{1/2}(\Sigma^*)\otimes\mathbb C^2,
\end{gather*}
where $H^{1/2}(\Sigma)$ denotes the space of holomorphic spinors on~$\Sigma$. Given a (measurable) loop $g\colon S\to {\rm SU}(2)$, there is an
associated unitary multiplication operator $M_g$ on $\Omega^{1/2}(S)\otimes\mathbb C^2$, and relative to the polarization
\begin{gather*}
M_g=\left(\begin{matrix}A(g)&B(g)\\C(g)&D(g)\end{matrix}\right).
\end{gather*}
In the classical case $A(g)$ ($B(g)$) is the classical block Toeplitz operator (Hankel operator, respectively),
associated to the symbol $g$. In general the `spin Toeplitz (Hankel) operators'~$A(g)$ ($B(g)$, respectively) have
many of the same qualitative properties
as in the classical case, because the projection $\Omega^{1/2}\to H^{1/2}$ dif\/fers from the classical projection
by a smoothing operator.

\begin{Theorem}\label{SU(2)theorem3} Suppose that $g\colon S\to {\rm SU}(2)$ $($is smooth and$)$ has a factorization as in Theo\-rem~{\rm \ref{SU(2)theorem2}}.
Then for any choice of spin structure for which $\overline{\partial}$ is invertible,
 \begin{gather*}\det\big(A(g)A\big(g^{-1}\big)\big)=\det\big(A(k_1)A\big(k_1^{-1}\big)\big)\det\big(\dot A(e^{\chi}) \dot A (e^{-\chi})\big)^2\det\big(A(k_2)A\big(k_2^{-1}\big)\big),
 \end{gather*}
where in the middle factor $\dot A(e^{\chi})$ is the compression to $H^{1/2}(\Sigma)$ of~$e^{\chi}$
as a $($scalar$)$ multiplication operator on $\Omega^{1/2}(S)$ $($which accounts for the square of this factor$)$.
\end{Theorem}

Even in the classical case, our proof of this is far more illuminating than the one in~\cite{PSIGMA}. For example
we will prove the following much stronger statement, at the level of operators:

\begin{Theorem}\label{SU(2)theorem4} Suppose that $k_1$, $k_2$ are measurable multiloops $S\to {\rm SU}(2))$ of the form
\begin{gather*}k_1=\left(\begin{matrix} a(z)&b(z)\\
-b^*(z)&a^*(z)\end{matrix} \right)\qquad \text{and} \qquad k_2=\left(\begin{matrix} d^{*}(z)&-c^{*}(z)\\
c(z)&d(z)\end{matrix} \right),
\end{gather*}
where $a$, $b$, $c$ and
$d$ are boundary values of holomorphic function in~$\Sigma$. Then \begin{gather*}A(k_1^*k_2)=A(k_1^*)A(k_2).\end{gather*}
\end{Theorem}

There are other senses in which the factors $k_1$, $\exp(\chi)$ and $k_2$ are expected to be ``independent''. For example
in the classical case we have previously conjectured that these factors are independent random variables with respect to the large
temperature limit for Wiener measure on the loop group, and that they Poisson commute with respect to the Evens--Lu
homogeneous Poisson structure on ${\rm LSU}(2)$ (see~\cite{PSIGMA}).
To properly formulate nonclassical analogues of these conjectures, we need to prove the existence of a factorization for a generic
loop, as we discussed at the end of the previous subsection.

\begin{Remark} \quad
\begin{enumerate}\itemsep=0pt
\item[(a)] One can replace the def\/ining representation $\mathbb C^2$ by any representation of ${\rm SU}(2)$ and there is a
corresponding factorization.
\item[(b)] $A(g)$ is a Fredholm operator. As elegantly explained by Quillen and Segal, $\det(A(g))$ can be def\/ined as a section of a determinant line bundle, or alternatively as a function on a~central extension of maps $S\to {\rm SU}(2)$ (see~\cite{PS}). In the classical case
Theorem \ref{SU(2)theorem3} can be ref\/ined to yield a factorization for $\det(A(g))$, but this is open in general.
\item[(c)] $\det(A(g))$ is essentially the state corresponding to $\Sigma$ for the ${\rm SU}(2)$ WZW model at level one. It is interesting to ask whether there might be some factorization for the corresponding state at higher level. This is true in the classical case, but for a trivial reason: the state at level~$l$ is the $l$th power of the state at level one, i.e., one replaces the def\/ining representation by $(\mathbb C^2)^{\otimes l}$.
\end{enumerate}
\end{Remark}

\subsection{Plan of the paper}

In Section~\ref{LTF} we use ideas of Krichever and Novikov to obtain a linear
triangular factorization, as in~(\ref{LinearTF}), for a reasonable function $f\colon S \to\mathbb C$ (in fact for any
hyperfunction). In Section~\ref{SU(2)thm1} we prove Theorem~\ref{SU(2)theorem1} and a technical ref\/inement for Holder continuous loops.
In Section \ref{thm1refinement} we will present a generalization of Theorem \ref{SU(2)theorem1} in which we relax the simultaneously nonvanishing hypotheses in~(I.1) and~(II.1). In Sections~\ref{SU(2)thm2} and~\ref{SU(2)thm3},
we prove Theorems~\ref{SU(2)theorem2} and~\ref{SU(2)theorem3}, respectively. In Section~\ref{generic} we
further discuss factorization and semistable bundles. In Section~\ref{scalarcase}
we present some calculations of determinants for spin Toeplitz operators in the scalar case (the middle factor in Theorem~\ref{SU(2)theorem3}); one suspects this can be substantially improved. Finally there is Appendix~\ref{hyperfunctions} on hyperfunctions, which we use sporadically in the text.

In part~II of this paper~\cite{BP2}, we will present a number of explicit calculations, especially for elliptic and hyperelliptic surfaces. A~motivating question is whether it is possible to view spin Toeplitz operators in this paper as deformations of
classical Toeplitz operators by allowing the surface to degenerate.

\section{Linear triangular factorization}\label{LTF}

In the remainder of the paper, $\Sigma$ denotes a connected compact Riemann surface with nonempty
boundary~$S$. Let $\widehat{\Sigma}$
denote the double, i.e.,
\begin{gather*}\widehat{\Sigma}=\Sigma^* \circ \Sigma,\end{gather*}
where $\Sigma^*$ is the adjoint of $\Sigma$, the surface $\Sigma$
with the orientation reversed, and the composition is sewing along
the common boundary $S$. Let $R$ denote the antiholomorphic
involution (or ref\/lection) f\/ixing~$S$.

The singular cohomology group $H^1(\widehat{\Sigma},\mathbb R)$ is a real symplectic vector space with respect to wedge product.
There is a positive polarization of the complexif\/ication
\begin{gather*}H^1\big(\widehat{\Sigma},\mathbb C\big)=H^{1,0}\oplus H^{0,1},\end{gather*}
where we concretely identify $H^{1,0}$ with the complex $\text{genus}(\widehat{\Sigma})$-dimensional vector space
of holomorphic dif\/ferentials on $\widehat{\Sigma}$,
and~$H^{0,1}$ is the conjugate vector space of antiholomorphic dif\/feren\-tials.
In this way we can think of $H^{1,0}$ as the dual of $ H^{0,1}$ and vice versa.
There is also a~Dolbeault isomorphism
\begin{gather*}
H^1\big(\mathcal O\big(\widehat{\Sigma}\big)\big) \to H^{0,1}\big(\widehat{\Sigma}\big),
\end{gather*}
where $\mathcal O(\widehat{\Sigma})$ denotes the
sheaf of holomorphic functions (see~\cite[p.~45]{GH}, but we will spell out what we need in~(\ref{connectmap})).

Let $H^0(S)$ denotes the space of analytic functions on~$S$ (hence each such function can be extended to a holomorphic function in
an open neighborhood of $S$ in $\widehat{\Sigma}$).
Using Mayer--Vietoris for the sheaf of holomorphic functions, corresponding to a slight thickening of the covering
of $\widehat{\Sigma}$ by $\Sigma$ and $\Sigma^*$, there is an exact sequence of vector spaces
\begin{gather*}0\to H^0\big(\widehat{\Sigma}\big) \to H^0(\Sigma)\oplus H^0(\Sigma^*)
\to H^{0}(S) \to H^{0,1}\big(\widehat{\Sigma}\big)\to 0.
\end{gather*}
In terms of the Dolbeault isomorphism, the connecting map $ H^{0}(S) \to H^{0,1}(\widehat{\Sigma})$ is realized concretely as a map
\begin{gather}\label{connectmap}H^{0}(S) \to H^{0,1}\big(\widehat{\Sigma}\big)\colon\
 f \to \begin{cases}\overline\partial f_1& \text{in} \ \Sigma,\\
 \overline{\partial}f_2&\text{in} \ \Sigma^*,
 \end{cases} \end{gather} where in the $C^{\infty}$ category,
$f=f_1-f_2$, $f_1$ is smooth in (a slight open enlargement of)~$\Sigma$, and~$f_2$ is smooth in (a~slight
open enlargement of)~$\Sigma^*$.

In reference to the following statements, recall that the space of complex hyperfunctions on an oriented
analytic one-manifold is the dual
of the space of complex analytic one forms on the one-manifold, and is denoted by $\operatorname{Hyp}(S)$;
see Appendix~\ref{hyperfunctions} (and~\cite{Sato} for the original reference).
The introduction of hyperfunctions is
possibly a distraction at this point of the paper. The point is that eventually we will want to
consider `transition functions' with minimal regularity; all such functions can be viewed as
hyperfunctions, and the natural domain for connecting maps, such as~(\ref{connectmap}), is a space of hyperfunctions.

\begin{Lemma} The connecting map \eqref{connectmap}
is dual to the injective map
\begin{gather*}
C^{\omega}\Omega^1(S,\mathbb C)\leftarrow H^{1,0}\big(\widehat{\Sigma}\big)\equiv H^1\big(\widehat{\Sigma},\mathbb R\big)\leftarrow 0
\end{gather*}
given by restriction of a holomorphic differential to~$S$. Consequently the
connecting map extends continuously to a surjective map
\begin{gather*}
\operatorname{Hyp}(S) \to H^{0,1}\big(\widehat{\Sigma}\big)\equiv H^1\big(\widehat{\Sigma},\mathbb R\big)\to 0.
\end{gather*}
\end{Lemma}

\begin{proof} We f\/irst show that the connecting map is dual to the restriction map.
Suppose that $f \in H^0(S)$ (i.e., functions which are holomorphic in a neighborhood of $S$) and write $f=f_1-f_2$
as in the realization of the connecting map (\ref{connectmap}), $f\to \overline{\partial}f_{j}$. Suppose also that
$\omega\in H^{1,0}(\widehat{\Sigma}) $. To prove the formula for the transpose, we must show that
\begin{gather*}
(\overline \partial f_{j},\omega)=(f,\omega\vert_S).
\end{gather*}
The left-hand side equals (using $d\omega=0$ and Stokes's theorem)
\begin{gather*}
\int_{\Sigma}\overline \partial f_1\wedge \omega +\int_{\Sigma^*}\overline \partial f_2\wedge \omega=\int _Sf_1
\wedge \omega-\int_Sf_2\wedge \omega=\int_S f\wedge \omega
\end{gather*} and this equals the right-hand side.

The fact that the connecting map has a continuous extension to hyperfunctions follows from the fact that it is the dual
of the restriction map, which is continuous. The extension is described more explicitly in Appendix~\ref{hyperfunctions}. The basic
idea (due to Sato) is that a hyperfunction $f$ on $S$ can be represented as a pair, which we heuristically write
as $f=f'+f''$, where $f'$ is holomorphic in an annular region just outside $S$ and $f''$ is holomorphic in an annular
region just inside $S$. We then compute the connecting map using three open sets, the interior of $\Sigma$, a suf\/f\/iciently
small annular region containing $S$, and the interior of $\Sigma^*$. One must check that this is independent of
the choice of representation for the hyperfunction: we can also represent $f=(f'+g)+(f''-g)$, where $g$ is holomorphic
in a small annular region containing~$S$.
\end{proof}

Fix a basepoint $(0)\in \Sigma^0:=\text{interior}(\Sigma)$, and let $(\infty)\in \Sigma^{*}$
denote the ref\/lected basepoint. The starting point for a series by papers by Krichever and Novikov
(see for example~\cite{KN}, which summarizes several of their papers) is the following

\begin{Lemma}There exists a unique meromorphic differential $dk$ on $\widehat{\Sigma}$ which is holomorphic in
$\widehat{\Sigma}{\setminus}\{(0),(\infty)\}$, has simple poles at $(0),(\infty)$ with
residues $\pm 1$, respectively, and satisfies $\operatorname{Re}(\int_{\gamma}dk)$ $=0$,
for all closed loops in $\widehat{\Sigma}{\setminus}\{(0),(\infty)\}$.
\end{Lemma}

\begin{proof} Let $\kappa$ denote the canonical line bundle, and $L$ the line bundle
corresponding to the divisor $(0)+(\infty)$. The holomorphic sections of the line bundle $\kappa\otimes L$
are meromorphic dif\/ferentials with at most simple poles at $(0),(\infty)$. The Riemann--Roch theorem
implies that \begin{gather*}
\dim\big(H^0(\kappa\otimes L)\big)-\dim\big(H^1(\kappa\otimes L)\big)=\deg(\kappa\otimes L)-(g-1),
\end{gather*}
where $g=\operatorname{genus}(\widehat{\Sigma})$. But
\begin{gather*}
\dim\big(H^1(\kappa\otimes L)\big)=\dim\big(H^0\big(\kappa\otimes (\kappa\otimes L)^{-1}\big)\big)=\dim\big(H^0\big(L^{-1}\big)\big)=0.
\end{gather*}
Hence the dimension of the space of meromorphic dif\/ferentials with at most simple poles at~$(0),(\infty)$
is $g+1$. A~dif\/ferential cannot have a single pole. So any meromorphic dif\/ferential in this space, which
is not globally holomorphic, will have
simple poles at~$(0)$ and $(\infty)$ with the sum of the residues necessarily equal to zero.
The normalizations in the statement of the theorem uniquely determine~$dk$.
\end{proof}

In the classical case $dk=dz/z$.
Following Krichever and Novikov, set $\tau=\operatorname{Re}(k)$, which is a~single valued harmonic function on~$\widehat{\Sigma}{\setminus}\{(0),(\infty)\}$, which limits
to~$-\infty$ at~$(0)$ and $+\infty$ at $(\infty)$; this is thought of as a distinguished time parameter.

The following summarizes a ``linear triangular decomposition'' for functions which follows immediately from the
existence of $dk$. The point is that the existence of $dk$, which induces a~metric on~$S$,
enables us to f\/ind a complement to the sum of the
two subspaces, holomorphic functions in the interior of $\Sigma$, and holomorphic functions in the interior of~$\Sigma^*$.

\begin{Proposition}\label{hyperfunction} A hyperfunction $\chi\colon \widehat{\Sigma}\to\mathbb C$ can be written uniquely
\begin{gather*}
\chi=\chi_- + \chi_0 +\chi_+,
\end{gather*}
where $\chi_-\in H^0(\Sigma^{0*},(\infty);\mathbb C,0)$, $\chi_+\in H^0(\Sigma^0,(0);\mathbb C,0)$, and
$\chi_0 dk$ is in the span of $H^{1,0}(\widehat{\Sigma})$ and~$dk$.
\end{Proposition}

\begin{Remark}\label{ltfcomments} \quad
\begin{enumerate}\itemsep=0pt
\item[(a)] The zeros of $dk$ are the critical points for the time parameter $\tau$.
For $dk$ the number of zeros is $2\text{genus}(\widehat{\Sigma})$, because there are two poles,
and the degree of $\kappa$ is $2\text{genus}(\widehat{\Sigma})-2$. The zeros, are located at the singular points
for the time slices of the surface.

\item[(b)] From (a) it follows that the set of $\chi_0$ can be characterized in the following way:
$\chi_0$ can have at most simple poles at the simple zeros of~$dk$, and $\chi_0(0)=\chi_0((\infty))$. (To explain the last condition, note that $\chi_0dk=c_0dk+\text{holomorphic dif\/ferential}$,
and hence $\chi_0-c_0$ must vanish at $(0),(\infty)$.)

\item[(c)] Since the complementary functions $\chi_0$ can have poles at points other than the base\-points~$(0)$ and~$(\infty)$, they
do not necessarily belong to the algebra $\mathcal M^0$ of meromorphic functions which are holomorphic in
$\widehat{\Sigma}{\setminus}\{(0),(\infty)\}$. Consequently our choice of complementary subspace is not the same
as the implicit choice made by Krichiver and Novikov (this refers to the exceptional cases~(1.5) after Lemma~1 in~\cite{KN}). For emphasis,
note our choice does not depend on the general position hypothesis for~$(0)$ in~\cite{KN}.

\item[(d)] Suppose that $\Sigma$ has genus zero. In this case there is an alternative splitting, where $\chi_0$ is locally
constant on~$S$. This has the advantage that there is no dependence on the choice of basepoint. We will consider this elsewhere.
\end{enumerate}
\end{Remark}

For smooth functions denote the decomposition in Proposition~\ref{hyperfunction} by
\begin{gather*}
\Omega^0(S)= H_+\oplus H_0\oplus H_-.
 \end{gather*}

\begin{Lemma}\label{stabilitybymultiply} The action of multiplication by $F\in H^0(\Sigma)$ maps $H_0\oplus H_+$ into itself.
\end{Lemma}

\begin{proof} First note it is obvious that multiplication by $F$ maps $H_+$ into itself.

Suppose that $\chi_0\in H_0$. Let $\chi_0^{(j)}$ denote a set of functions in $H_0$ such that
$\omega_j:=\chi_0^{(j)}dk$ is a~basis for the space of holomorphic dif\/ferentials on $\widehat{\Sigma}$.
We must show that there exist constants $c_1,\dots,c_{\text{genus}(\widehat{\Sigma})}$
such that (we write these as functions of a parameter $z$, to clarify the notation)
\begin{gather*}
F(z)\chi_0(z) -\sum_{j=1}^{\text{genus}(\widehat{\Sigma})} c_j \chi_0^{(j)}(z) \in H^0(\Sigma),
\end{gather*}
i.e., this function does not have any poles in~$\Sigma$.
Since the set of poles of~$F\chi_0$ (which are all simple) is a subset of the set of zeros
of~$dk$, this is equivalent to showing the following.
Let $\{z_1^0,\dots,z^0_{\text{genus}(\widehat{\Sigma})}\}$ denote the set of zeros of $dk$
in $\Sigma$ (see~(b) of Remark~\ref{ltfcomments}). Then there exist unique constants $c_1,\dots,c_{\text{genus}(\widehat{\Sigma})}$
such that
\begin{gather*}
F(z)\chi_0(z) -\sum_{j=1}^{\text{genus}(\widehat{\Sigma})} c_j \chi_0^{(j)}(z)
\end{gather*}
does not have any poles at the points $z^0_1,\dots,z^0_{\text{genus}(\widehat{\Sigma})}$.

To turn this into a system of scalar equations, for each $1\le i\le \text{genus}(\widehat{\Sigma})$ choose a local coordinate $z_i$
with $z_i(z_i^0)=0$. Relative to this choice of coordinates, we can then speak of residues. Our claim is then equivalent to
the claim that the $\text{genus}(\widehat{\Sigma})\times \text{genus}(\widehat{\Sigma})$ matrix
$\operatorname{res}^j_i:=\operatorname{Res}\big(\chi_0^{(j)},z^0_i\big)$ is nonsingular.

In the vicinity of $z^0_i$,
\begin{gather*}\chi_0^{(j)}=\frac{\operatorname{res}^j_i}{z_i}+\text{holomorphic}(z_i), \qquad
dk=\big(f_1^{(i)}z_i+O\big(z_i^2\big)\big)dz_i,
\end{gather*}
where the f\/irst coef\/f\/icient $f_1^{(i)}$ is not zero, and
\begin{gather*} \chi_0^{(j)} dk= \big(\operatorname{res}_i^j f_1^{(i)}+O(z_i)\big)dz_i.
\end{gather*}
We can just as well show that the matrix
$\operatorname{res}_i^j f_1^{(i)}$ is nonsingular. In invariant terms we must show that the evaluation map
\begin{gather*}
H^{1,0}(\widehat{\Sigma})\to \bigoplus_{i=1}^{\text{genus}(\widehat{\Sigma})}T^*_{z_i}
\end{gather*}
is an isomorphism of vector spaces. A holomorphic dif\/ferential in the kernel of this map
would then have $2\text{genus}(\widehat{\Sigma})$ zeros and no poles, which is impossible
since the degree of $\kappa$ is \mbox{$2\text{genus}(\widehat{\Sigma})-2$}.
\end{proof}

In the classical case it is a (dif\/f\/icult to prove but) well-known fact that if
$\chi$ is Holder continuous of order $s>0$, and $s$ is nonintegral,
then $\chi_{\pm}$ are Holder continuous of order $s$. This is equivalent to the fact that the Hilbert transform
$\mathcal H:=iP_+-iP_-\colon \chi\to i\chi_+-i\chi_-$, which can be expressed as a principal value integral relative to the Cauchy kernel
\begin{gather}\label{Cauchy}\chi(z) \rightarrow
\frac{1}{\pi}p.v.\int_{S^1}\frac{\chi(\zeta)}{\zeta-z}d\zeta\end{gather}
is a bounded operator on $C^s(S^1)$, the Banach space of Holder continuous functions of order~$s$
(the original reference in the classical case is~\cite[Chapter~III, Section~3]{GK};
this is cited for example in \cite[p.~60]{CG}).

\begin{Theorem}\label{projcontinuity} Suppose that $s>0$ and nonintegral. In reference to
Proposition~{\rm \ref{hyperfunction}}, if $\chi\in C^s(S)$ $($i.e., Holder continuous of order $s)$,
then $\chi_{\pm}\in C^s(S)$.
\end{Theorem}

This is equivalent to proving the continuity of an analogue of the Hilbert transform on
$C^s(S)$. The kernel is essentially the Szego kernel, which in general is the classical Cauchy integral
operator plus a smooth perturbation
(see~\cite{HS}, and~\cite{BP2} for examples). Thus Theorem~\ref{projcontinuity} follows from
the classical case.

\section{Proof of Theorem \ref{SU(2)theorem1}}\label{SU(2)thm1}

In Theorem~\ref{SU(2)theorem1} it is obvious that~(ii) implies~(iii). This actually follows from the classical example,
because we can use $\mathfrak z$ to pull results from the sphere back to $\widehat{\Sigma}$.
It is also clear that~(iii) implies~(i): This follows immediately by just multiplying the factors together.
In the course of completing the proof of Theorem~\ref{SU(2)theorem1}, we will also
prove the following

\begin{Theorem}\label{SU(2)theorem1smooth} Suppose that $k_1 \in C^s(S,{\rm SU}(2))$,
where $s>0$ and nonintegral. The following are equivalent:
\begin{enumerate}\itemsep=0pt
\item[{\rm (I.1)}] $k_1$ is of the form
\begin{gather*}
k_1(z)=\left(\begin{matrix} a(z)&b(z)\\
-b^*(z)&a^*(z)\end{matrix} \right),\qquad z\in S,
\end{gather*} where $a,b\in
H^0(\Sigma)$ have $C^s$ boundary values, $a((0))>0$, and $a$ and $b$
do not simultaneously vanish at a point in $\Sigma$.

\item[{\rm (I.3)}] $k_1$ has a factorization of the form
\begin{gather*}\left(\begin{matrix} 1&0\\
y^*(z)+y_0(z)&1\end{matrix} \right)\left(\begin{matrix} a_1&0\\
0&a_1^{-1}\end{matrix} \right)\left(\begin{matrix} \alpha_1 (z)&\beta_1 (z)\\
\gamma_1 (z)&\delta_1 (z)\end{matrix} \right),\end{gather*} where $y\in H^0(\Sigma^0)$, $y_0$ is a zero mode, $a_1>0$, the last factor is in
$H^0(\Sigma^0,{\rm SL}(2,\mathbb C))$ and is unipotent upper triangular
at the basepoint $(0)$, and the factors
have $C^s$ boundary values.
\end{enumerate}

Similarly, the following are equivalent:
\begin{enumerate}\itemsep=0pt
\item[{\rm (II.1)}] $k_2$ is of the form
\begin{gather*}k_2(z)=\left(\begin{matrix} d^{*}(z)&-c^{*}(z)\\
c(z)&d(z)\end{matrix} \right),\qquad z\in S^1,\end{gather*} where $c,d\in
H^0(\Sigma)$ have $C^s$ boundary values, $c((0))=0$, $d((0))>0$, and
$c$ and $d$ do not simultaneously vanish at a point in $\Sigma$.

\item[{\rm (II.3)}] $k_2$ has a factorization of the form
\begin{gather*}
\left(\begin{matrix} 1&x^*(z)+x_0(z)\\
0&1\end{matrix} \right)\left(\begin{matrix} a_2&0\\
0&a_2^{-1}\end{matrix} \right)\left(\begin{matrix} \alpha_2 (z)&\beta_2 (z)\\
\gamma_2 (z)&\delta_2 (z)\end{matrix} \right),
\end{gather*} where $x\in H^0(\Sigma^0,(0);\mathbb C,0)$,
$x_0$ is a zero mode, $x_0((0))=0$ $a_2>0$, the last factor is
in $H^0(\Sigma^0,{\rm SL}(2,\mathbb C))$ and is upper triangular
unipotent at the basepoint $(0)$, and the
factors have $C^s$ boundary values.
\end{enumerate}

The triangular factorizations in {\rm (I.3)} and~{\rm (II.3)} are uniquely determined.
\end{Theorem}

\begin{Remark} When $k_2$ is the restriction to $S$ of a function with entries in $\mathcal M^0$, the determinant condition
$c^*c+d^*d=1$ can be interpreted as an equality of functions in~$\mathcal M^0$. Together with $d((0))>0$, this implies
that $c$ and $d$ do not simultaneously vanish. Thus the simultaneous vanishing
hypotheses in~(I.1) and~(II.1) of Theorem~\ref{SU(2)theorem1smooth} are superf\/luous in that case.
\end{Remark}

\begin{proof} The two sets of conditions are proven in the same way. It is obvious that~(II.3) implies~(II.1).
The dif\/f\/icult task is to show that~(II.1) implies~(II.3).

Consider the polarization
\begin{gather*}
\Omega^0\big(S,\mathbb C^2\big)= \mathcal H_+\otimes \mathbb C^2\oplus \mathcal H_-\otimes C^2,
\end{gather*}
where $\mathcal H_+=H_++\mathbb C$ (functions which are holomorphic in~$\Sigma$),
and $\mathcal H_-$ is the sum of $H_-$ and the subspace of zero modes $x_0\in H_0$
such that $x_0dk$ is a global holomorphic dif\/ferential (i.e., $x_0((0))=0$,
as in the statement of~(II.3)). We write $P_{\pm}$ for the corresponding projections.
We will view $\Omega^0(S,\mathbb C^2)$ as a preHilbert space, by using
the measure induced by the restriction of $(2\pi i)^{-1}dk$ to~$S$. (Note that for a dif\/ferential $\omega$,
if $\omega^*:=\overline{R^*\omega}$, then $(dk)^*=-dk$, as follows by checking the asymptotics at $(0)$ and $(\infty)$.
Thus along $S$, $(2\pi i)^{-1}dk$ is real. It is also nonvanishing; see~(a) of Remark~\ref{ltfcomments}.
The polarization is an orthogonal direct sum.)

Relative to this polarization,
we write the unitary multiplication operator corresponding to a multiloop $g\colon S\to {\rm SU}(2)$ as
\begin{gather*}
M_g=\left(\begin{matrix}A&B\\C&D\end{matrix}\right).
\end{gather*}
For a smooth loop the of\/f diagonals (Hankel type operators) will be small and $A$ and $D$ will be Fredholm.

We must show that the multiloop $k_2$ has a unique factorization as
in~(II.3), i.e., we must solve for $a_2$, $x^*$, and so on, in
\begin{gather*}
\left(\begin{matrix}d^*&-c^*\\c&d\end{matrix}\right)=\left(\begin{matrix}1&x^*+x_0\\0&1\end{matrix}\right)
\left(\begin{matrix}a_2&0\\0&a_2^{-1}\end{matrix}\right)
\left(\begin{matrix}\alpha_2&\beta_2\\\gamma_2&\delta_2\end{matrix}\right).
\end{gather*}
The second row implies
\begin{gather*}(c,d)=a_2^{-1}(\gamma_2,\delta_2).
\end{gather*}
This determines $a_2^{-1}\gamma_2$ and $a_2^{-1}\delta_2$, and (using the unipotence of the third factor in~(II.3)),
\mbox{$a_2^{-1}=d((0))$}.

The f\/irst row implies
\begin{gather}\label{eqn1}
d^*=a_2\alpha_2+(x^*+x_0)c,\qquad -c^*=a_2\beta_2+(x^*+x_0)d.
\end{gather}
The f\/irst term on the right-hand side of both of these equations belongs to~$\mathcal H_+$. To
solve for~$x^*+x_0$, we will apply a projection to get rid of these~$\mathcal H_+$ terms.

Consider the operator
\begin{gather*}
T\colon \ \mathcal H_-\to \mathcal H_-\oplus \mathcal H_-\colon \ x_0+x^{*}\to ((c(x_0+x^{*}))_{-},(d(x_0+x^{*}))_{
-}),
\end{gather*} where $(\cdot)_{-}$ is the projection to $\mathcal H_-$. To show that we can uniquely solve for
$x^*+x_0$, we will show $T$ is injective and that $((d^*)_{-},(-c^*)_{-})$ is in the range of~$T$.

The operator $T$ is the restriction of the Fredholm operator $D(k_2)^*=D(k_2^*)$
to the subspace $\{(x_0+x^*,0)\}$, consequently the image of $T$ is closed. $T$ is also injective.
For suppose that both $((c(x_0+x^{*}))_{-}$ and $(d(x_0+x^{*}))_{
-}$ vanish. Then
\begin{gather*}
c(x_0+x^{*})=g \qquad \text{and} \qquad d(x_0+x^{*})=h,
\end{gather*}
where $g,h\in \mathcal H_+$. Since $c$, $d$ do not simultaneously vanish
in~$\Sigma$, this implies that $x_0+x^*\in \mathcal H_+$. But this means that $x_0+x^*$ must vanish.
Thus~$T$ is injective and has a closed image.

The adjoint of $T$ is given by
\begin{gather*}T^{*}\colon \ \mathcal H_-\oplus \mathcal H_-\to \mathcal H_-\colon \ (f_0+f^{*},g_0+g^{*})\to  (c^{*}(f_0+f^{*})+d^{
*}(g_0+g^{*}) )_{-}.
\end{gather*}

If $(f_0+f^{*},g_0+g^{*})\in \operatorname{ker}(T^{*})$, then
\begin{gather*}
c^{*}f^{*}+d^{*}g^{*}+ (c^{*}f_0+d^{
*}g_0 )_{-}=0.
\end{gather*}
By Lemma \ref{stabilitybymultiply}
\begin{gather*}
(c^{*}f_0+d^{*}g_0 )_{-}=c^{*}f_0+d^{*}g_0.
\end{gather*}
Thus
\begin{gather*}
c^{*}(f_0+f^{*})+d^{*}(g_0+g^{*})=0
\end{gather*} viewed as a meromorphic function in~$\Sigma^*$,
vanishes in the closure of $\Sigma^{*}$. Because $\vert c\vert^2+\vert d\vert^ 2=1$ around~$S$,
\begin{gather}\label{keyequality}
(f_0+f^{*},g_0+g^{*})=\lambda^{*}(d^{*},-c^{*}),
\end{gather} where $\lambda^{*}$ is
meromorphic in $ \Sigma^{*}$ and vanishes at $(\infty)$ because
$d^{*}((\infty))=d((0))>0$. Thus $\lambda$ is meromorphic in $\Sigma$ and vanishes at $(0)$.

We now claim that $((d^{*})_{-},-c^{*})\in
\operatorname{ker}(T^{*})^{\perp}$. To prove this, suppose that $(f_0+f^{*},g_0+g^{*})\in \operatorname{ker}(T^{*})$, as in
the previous paragraph. Then
 \begin{gather*}\int_S ((d^{*})_{-}(f_0^*+f)+(-c^{*})(g_0^*+g))dk
=\int_S\lambda (d^{*}d+c^{*} c)dk =\int_S\lambda dk.
 \end{gather*} (Note that constants are orthogonal to $\mathcal H_-$, and hence
we could replace~$(d^*)_-$ by~$d^*$.) We claim this integral vanishes.
Since~$\lambda dk$ is a meromorphic dif\/ferential in $\Sigma$ (by the previous paragraph), this equals the sum of
residues of $\lambda dk$ in $\Sigma$. Because $f_0^* dk$ and $g_0^* dk$ are holomorphic dif\/ferentials (the def\/ining
characteristic for zero modes),
and $c$ and $d$ do not simultaneously vanish, the only point we need to worry about is $(0)$ (see~(\ref{keyequality})).
Finally the residue at $(0)$ is zero, because $\lambda((0))=0$ (by the previous paragraph).

Because $T$ has closed image, there
exists $x_0+x^{*}\in \mathcal H_{-}$ such that
\begin{gather*}
(d^{*})_{-}=((x_0+x^{*})c)_{-}\qquad \text{and}\qquad
-c^{*}=((x_0+x^{*})d)_{-}.\end{gather*} We can now solve for $a_2\alpha_2$
and $a_2\beta_2$ in~(\ref{eqn1}). We previously noted that $a_2^{-1}=d((0))$.
Now that we have solved for the factors, the form of the factorization immediately
implies $\alpha_2\delta_2-\gamma_2\beta_2=1$
on~$S$; by holomorphicity of the terms, this also holds in $\Sigma$.

Theorem~\ref{projcontinuity} implies that the projections $P_{\pm}$ are continuous on~$C^s(S)$. Consequently the operator $T$ will be continuous and have the same properties on
the $C^s$ completions. Hence when $k_2\in C^s$, the factors are $C^s$. This
completes the proofs of Theorems~\ref{SU(2)theorem1} and~\ref{SU(2)theorem1smooth}.
\end{proof}

\subsection{A generalization of Theorem~\ref{SU(2)theorem1}}\label{thm1refinement}

In the preceding proof a key step is to show that $\lambda dk$ is a holomorphic dif\/ferential
in $\Sigma$. The proof is not sharp, in the sense that $\lambda$ itself is holomorphic,
whereas for the proof to go through, we can allow $\lambda$ to have simple poles at the zeros of $dk$.
By considering~(\ref{keyequality}), this suggests
that there ought to be a generalization in which $c$ and $d$ are allowed to have simultaneous zeros to f\/irst
order, within the zero set of $dk$. This in turn forces $a_2$ to be a function. It turns out to be natural
for $a_1$ and $a_2$ to be functions, for other reasons, as we will explain in Section~\ref{generic}.

\begin{Theorem}\label{SU(2)theorem1smooth2} Suppose that $k_1 \in C^s(S,{\rm SU}(2))$,
where $s>0$ and nonintegral. The following are equivalent:
\begin{enumerate}\itemsep=0pt
\item[{\rm (I.1)}] $k_1$ is of the form
\begin{gather*}
k_1(z)=\left(\begin{matrix} a(z)&b(z)\\
-b^*(z)&a^*(z)\end{matrix} \right),\qquad z\in S,
\end{gather*} where $a,b\in
H^0(\Sigma)$ have $C^s$ boundary values, $a((0))>0$, and~$a$ and $b$
can simultaneously vanish only to first order, and this set of common zeros is the set of
poles in $\Sigma$ of a zero mode.

\item[{\rm (I.3)}] $k_1$ has a factorization of the form
\begin{gather*}
\left(\begin{matrix} 1&0\\
y^*(z)+y_0(z)&1\end{matrix} \right)\left(\begin{matrix} a_1(z)&0\\
0&a_1(z)^{-1}\end{matrix} \right)\left(\begin{matrix} \alpha_1 (z)&\beta_1 (z)\\
\gamma_1 (z)&\delta_1 (z)\end{matrix} \right),
\end{gather*}
where $a_1^{-1}$ is a nonvanishing sum
of a zero mode and a holomorphic function in $\Sigma$, $a_1((0))>0$,
$y\in H^0(\Sigma^0)$, the last factor is in
$H^0(\Sigma^0,{\rm SL}(2,\mathbb C))$ and is unipotent upper triangular
at the basepoint $(0)$, and the factors
have $C^s$ boundary values.
\end{enumerate}

Similarly, the following are equivalent:
\begin{enumerate}\itemsep=0pt
\item[{\rm (II.1)}] $k_2$ is of the form
\begin{gather*}k_2(z)=\left(\begin{matrix} d^{*}(z)&-c^{*}(z)\\
c(z)&d(z)\end{matrix} \right),\qquad z\in S^1,\end{gather*} where $c,d\in
H^0(\Sigma)$ have $C^s$ boundary values, $c((0))=0$, $d((0))>0$, and $c$ and $d$
can simultaneously vanish only to first order, and this set of common zeros is the
set of poles in~$\Sigma$ of a~zero mode.

\item[{\rm (II.3)}] $k_2$ has a factorization of the form
\begin{gather*}
\left(\begin{matrix} 1&x^*(z)+x_0(z)\\
0&1\end{matrix} \right)\left(\begin{matrix} a_2(z)&0\\
0&a_2(z)^{-1}\end{matrix} \right)\left(\begin{matrix} \alpha_2 (z)&\beta_2 (z)\\
\gamma_2 (z)&\delta_2 (z)\end{matrix} \right),
\end{gather*}
 where $a_2$ is a nonvanishing sum
of a zero mode and a holomorphic function in $\Sigma$, \mbox{$a_2((0)){>}0$},
$x\in H^0(\Sigma^0,(0);\mathbb C,0)$, $x_0$ is a zero mode, the last factor is
in $H^0(\Sigma^0,{\rm SL}(2,\mathbb C))$ and is upper triangular
unipotent at the basepoint $(0)$, and the
factors have $C^s$ boundary values.
\end{enumerate}
\end{Theorem}

\begin{Remark}\label{exist2remarks}\quad
\begin{enumerate}\itemsep=0pt
\item[(a)] In contrast to Theorem \ref{SU(2)theorem1}, we are not asserting that the triangular
factorizations in~(I.3) and~(II.3) are uniquely determined. In particular, as the theorem is stated, $a_1$~and~$a_2$ can be multiplied
by nonvanishing holomorphic functions in $\Sigma$.

\item[(b)] In the preceding theorem $a_1^{-1}$ and $a_2$ are not generally zero modes. Nonconstant zero modes always vanish at some point in $\Sigma$,
and this would mean that for example~$c$ or~$d$ would
have a pole. We want to avoid poles, because for example the results in Section~\ref{SU(2)thm3} are dependent
on the assumption that~$a$,~$b$,~$c$,~$d$ are holomorphic in~$\Sigma$.
\end{enumerate}
\end{Remark}

\begin{proof} The two sets of conditions are proven in the same way. It is obvious that~(II.3) implies~(II.1).
As before, the dif\/f\/icult task is to show that~(II.1) implies~(II.3).

Assuming~(II.1), we must show that the multiloop $k_2$ has a factorization as
in~(II.3), i.e., we must prove existence of $a_2$, $x^*$, and so on, in
\begin{gather*}
\left(\begin{matrix}d^*&-c^*\\c&d\end{matrix}\right)=\left(\begin{matrix}1&x^*+x_0\\0&1\end{matrix}\right)
\left(\begin{matrix}a_2&0\\0&a_2^{-1}\end{matrix}\right)
\left(\begin{matrix}\alpha_2&\beta_2\\\gamma_2&\delta_2\end{matrix}\right).
\end{gather*}
The second row implies
\begin{gather*}
(c,d)=a_2^{-1}(\gamma_2,\delta_2).
\end{gather*}
This determines $a_2^{-1}\gamma_2$ and $a_2^{-1}\delta_2$. In the context of the previous subsection, i.e., when~$c$ and~$d$ do not simultaneously vanish, $a_2$
is a constant, determined by the unipotence of the third factor in~(II.3), which (in general) implies $a_2^{-1}((0))=d((0))$.
In general, towards determining~$a_2$, note that $a_2$ must have simple poles at
the common zeros for~$c$ and~$d$, and cannot have zeros in~$\Sigma$, so that~$\gamma_2$ and~$\delta_2$ are holomorphic and do not simultaneously vanish in~$\Sigma$. The hypotheses of~(II.1) guarantee that there is a zero mode with the appropriate (necessarily
simple) poles in~$\Sigma$. Fix a slight open enlargement of~$\Sigma$. Since this surface is open, we can f\/ind a nonvanishing meromorphic function in this enlargement with the same singular behavior as this zero mode; the dif\/ference between this meromorphic function,
which is our choice for~$a_2$, and the given zero mode is holomorphic in~$\Sigma$. As we remarked above, this choice is far from uniquely determined, and it is not clear how to pin down a preferred choice.

The f\/irst row implies
\begin{gather*}
d^*=a_2\alpha_2+(x^*+x_0)c,\qquad -c^*=a_2\beta_2+(x^*+x_0)d.
\end{gather*}
The basic complication, compared to the proof of Theorem~\ref{SU(2)theorem1}, is that $a_2\alpha_2 $ and $a_2\beta_2$ are
no longer necessarily holomorphic in~$\Sigma$, because~$a_2$ has poles. We will f\/irst solve (not necessarily uniquely)
for~$x^*$, by using the same
strategy as in the proof of Theorem~\ref{SU(2)theorem1}, but using a dif\/ferent polarization.

Consider the polarization
\begin{gather*}\Omega^0\big(S,\mathbb C^2\big)= \mathcal H_+\otimes C^2\oplus \mathcal H_-\otimes C^2,
\end{gather*}
where $\mathcal H_+$ is now the sum of functions which are holomorphic in $\Sigma$ and zero modes, and \mbox{$\mathcal H_-=H_-$}. We write~$P_+$, and $P_-$ for the projections
onto the subspaces $\mathcal H_+\otimes \mathbb C^2$ and $H_-\otimes C^2$, respectively.

Relative to this polarization,
we write the multiplication operator def\/ined by a multiloop as
\begin{gather*}M_g=\left(\begin{matrix}A&B\\C&D\end{matrix}\right).
\end{gather*}
For a smooth loop the of\/f diagonals (Hankel type operators) will be small and~$A$ and~$D$ will be Fredholm.

Consider the operator
\begin{gather*}
T\colon \ \mathcal H_-\to \mathcal H_-\oplus \mathcal H_-\colon \  x^{*}\to (((cx^{*})_{-},(dx^{*})_{
-}),
\end{gather*} where $(\cdot)_{-}$ is shorthand for the projection to $\mathcal H_-$.
The operator $T$ is the restriction of the Fredholm operator $D(k_2)^*=D(k_2^*)$
to the subspace $\{(x^*,0)\}$, consequently the image of $T$ is closed. Unfortunately in general $T$ is not injective.
For suppose that both $((cx^{*})_{-}$ and $(dx^{*})_{-}$ vanish. Then
\begin{gather*}cx^{*}=g \qquad \text{and} \qquad dx^{*}=h,
\end{gather*} where $g,h\in \mathcal H_+$. At a common simple zero for~$c$ and~$d$, it could happen
that~$g$,~$h$ also have a~common pole, so~$x^*$ does not necessarily have simple poles.

The adjoint of $T$ is given by
\begin{gather*}
T^{*}\colon \ \mathcal H_-\oplus \mathcal H_-\to \mathcal H_-\colon \ (f^{*},g^{*})\to \left(c^{*}f^{*}+d^{
*}g^{*}\right)_{-}.
\end{gather*}

If $(f^{*},g^{*})\in \operatorname{ker}(T^{*})$, then
\begin{gather*}c^{*}f^{*}+d^{*}g^{*}=0\end{gather*}
viewed as a meromorphic function in $\Sigma^*$. Because $\vert c\vert^2+\vert d\vert^ 2=1$ around~$S$,
\begin{gather*}
(f^{*},g^{*})=\lambda^{*}(d^{*},-c^{*}),
\end{gather*} where $\lambda^{*}$ is
meromorphic in $ \Sigma^{*}$ and vanishes at $(\infty)$ because
$d^{*}((\infty))=d((0))>0$. Thus $\lambda$ is meromorphic in~$\Sigma$ and vanishes at~$(0)$.

We now claim that $((d^{*})_{-},-c^{*})\in
\operatorname{ker}(T^{*})^{\perp}$. To prove this, suppose that $(f^{*},g^{*})\in \operatorname{ker}(T^{*})$, as in
the previous paragraph. Then
 \begin{gather*}
 \int_S ((d^{*})_{-}(f)+(-c^{*})(g))dk
=\int_S\lambda (d^{*}d+c^{*} c)dk =\int_S\lambda dk.
\end{gather*} We claim this integral vanishes.
Since $\lambda dk$ is a meromorphic dif\/ferential in $\Sigma$ (by the previous paragraph), this equals the sum of
residues of~$\lambda dk$ in~$\Sigma$. Because the poles of $\lambda$ occur at common zeros of~$c$,~$d$, and hence only at zeros of~$dk$,
and because $\lambda((0))=0$, $\lambda dk$ is holomorphic in $\Sigma$ (by the previous paragraph). Thus the integral vanishes.

Because $T$ has closed image, there
exists $x^{*}\in \mathcal H_{-}$ such that
\begin{gather*}
(d^{*})_{-}=(x^{*}c)_{-} \qquad \text{and}\qquad
-c^{*}=(x^{*}d)_{-}.\end{gather*}
(Here and below we are using the fact that $a_2\in H_0+H_+$, $\alpha_2\in H_+$,
and Lemma~\ref{stabilitybymultiply}, to conclude that $a_2\alpha_2\in H_0+H_+$, so that it is killed by~$(\cdot)_-$.)

Now consider the operator
\begin{gather*}
\mathcal T\colon \ (H_++\mathbb C)\oplus H_0\oplus H_+\rightarrow (H_++H_0)\oplus (H_++H_0)\colon\\
 \hphantom{\mathcal T\colon} \ (\alpha_2,x_0,\beta_2)\to (a_2\alpha_2+x_0c,a_2\beta_2+x_0d).
\end{gather*}
This is well-def\/ined by Lemma~\ref{stabilitybymultiply}.

We must show that
$(d^*-x_0c,-c^*-x_0d)$ (which is in the target of $\mathcal T$ by the previous paragraph) is in the image of~$\mathcal T$.
The operator~$\mathcal T$ has a closed image. The (Hilbert space) adjoint of~$\mathcal T$ is given by
\begin{gather}
(H_++H_0)\oplus (H_++H_0)\to (H_++\mathbb C)\oplus H_0\oplus H_+\colon \nonumber\\
 (f,g)\to ((a_2^*f)_{0+},(a_2^*g)_{0+},(c^*f+d^*g)_0).\label{adjointformula2}
\end{gather}

We need to show that $(d^*-x_0c,-c^*-x_0d)$ is orthogonal to the kernel of $\mathcal T^*$.
Suppose that $(f,g)\in \operatorname{ker}(\mathcal T^*)$ (so the three terms in~(\ref{adjointformula2}) vanish). Then
\begin{gather*}
(d^*-x_0c,f)+(-c^*-x_0d,g)=\text{const}\cdot \int_S(df^*-cg^*-x_0^*(c^*f+d^*g))dk.
\end{gather*}
The last two terms immediately drop out. Thus, up to a constant, this equals
\begin{gather*}
\int_S(df^*-cg^*)dk=\int_S\big(\big(da_2^{-1}\big)a_2f^*-\big(ca_2^{-1}\big)a_2g^*\big)dk.
\end{gather*}
The vanishing of the f\/irst two terms in~(\ref{adjointformula2}) imply that $a_2f^*,a_2g^*\in H_+$.
Thus the integral vanishes since the integrand is holomorphic in~$\Sigma$.

Once we have solved for the factors, the form of the factorization immediately implies $\alpha_2\delta_2-\gamma_2\beta_2=1$
on~$S$; by holomorphicity of the terms, this also holds in~$\Sigma$. This completes the proof.
\end{proof}

\section{Factorization and semistability}\label{SU(2)thm2}

\subsection{Proof of Theorem~\ref{SU(2)theorem2}}

We recall the statement of the theorem:

\begin{Theorem} Suppose $g\in C^{\infty}(S,{\rm SU}(2))$. If~$g$ has a factorization
\begin{gather*}
g(z)=k_1^*(z)\left(\begin{matrix} e^{\chi(z)}&0\\
0&e^{-\chi(z)}\end{matrix}\right)k_2(z),
\end{gather*} where $\chi \in
C^{\infty}(S,i\mathbb R)$, and $k_1$ and $k_2$ are as in {\rm (I.1)} and {\rm (II.1)}, respectively, of Theorem~{\rm \ref{SU(2)theorem1}},
then~$E(g)$, the holomorphic $G$ bundle on $\widehat{\Sigma}$
defined by~$g$ as a transition function, is semistable, and the associated bundle for the defining representation
has a sub-line bundle with an antiholomorphic reflection symmetry compatible with~$R$.
\end{Theorem}

\begin{Remark}\label{generality} In the following proof, $a_1$ is a positive constant and hence $a_1^*=a_1$. We will write~$a_1^*$ at various points in the proof because in the next subsection we will want to note that the calculations are valid more generally when~$a_1$ is a function. \end{Remark}

\begin{proof} In the proof we will use
the fact that, using a multiloop as a transition function, there is a bijective correspondence
between the double coset space
\begin{gather*}H^0\big(\Sigma^{0*},G\big) \backslash \operatorname{Hyp}(S,G)/H^0\big(\Sigma^0,G\big)
\end{gather*}
and the set of isomorphism classes
of holomorphic~$G$ bundles on~$\widehat{\Sigma}$, and that a smooth function def\/ines a hyperfunction; see Appendix~\ref{bundles}.

By assumption $k_1$ and $k_2$ have `triangular
factorizations' of the following forms:
\begin{gather*}
k_1=\left(\begin{matrix} 1&0\\
y^{*}+y_0&1\end{matrix} \right)\left(\begin{matrix} a_1&0\\
0&a_1^{-1}\end{matrix} \right)\left(\begin{matrix} \alpha_1&\beta_1\\
\gamma_1&\delta_1\end{matrix} \right),
\end{gather*} and
\begin{gather*}
k_2=\left(\begin{matrix} 1&x^{*}+x_0\\
0&1\end{matrix} \right)\left(\begin{matrix} a_2&0\\
0&a_2^{-1}\end{matrix} \right)\left(\begin{matrix} \alpha_2&\beta_2\\
\gamma_2&\delta_2\end{matrix} \right).
\end{gather*}

Given these `triangular factorizations' for~$k_1$
and~$k_2$, we can derive a~`triangular factorization' for $g$. (Note: we have not developed
a general theory of triangular factorization in the context of this paper, so we are using this term somewhat
loosely.) To simplify notation,
let $X=a_2^{-2}(x+x_0^*)$, $Y=a_1^2(y+y_0^*)$, and $a=a_1^*a_2$. Then
\begin{gather*}
g=\left(\begin{matrix} \alpha_1 &\beta_1\\
\gamma_1&\delta_1\end{matrix}
\right)^*\left(\begin{matrix}1&Y\\0&1\end{matrix}\right)\left(\begin{matrix}ae^{\chi_-+\chi_0+\chi_+}&0
\\0&(ae^{\chi_-+\chi_0+\chi_+})^{-1}\end{matrix}\right)
\left(\begin{matrix} 1&X^{*}\\
0&1\end{matrix} \right)\left(\begin{matrix} \alpha_2&\beta_2\\
\gamma_2&\delta_2\end{matrix} \right) \\
\hphantom{g}{}
=\left(\begin{matrix} \alpha_1^*&\gamma_1^*\\
\beta_1^*&\delta_1^*\end{matrix}
\right)\left(\begin{matrix}e^{\chi_-}&0
\\0&e^{-\chi_-}\end{matrix}\right)\left(\begin{matrix}1&e^{-2\chi_-}Y\\0&1\end{matrix}\right)
\left(\begin{matrix}ae^{\chi_0}&0
\\0&(ae^{\chi_0})^{-1}\end{matrix}\right)
\left(\begin{matrix} 1&e^{2\chi_+}X^{*}\\
0&1\end{matrix} \right)
\\
\hphantom{g=}{}\times \left(\begin{matrix}e^{\chi_+}&0
\\0&e^{-\chi_+}\end{matrix}\right)\left(\begin{matrix} \alpha_2&\beta_2\\
\gamma_2&\delta_2\end{matrix} \right).
\end{gather*}
The bundle $E(g) $ def\/ined by $g$ as a transition
function depends only on the product of the middle
three factors, because of the remark at the beginning of the proof. The product of the middle three factors equals
\begin{gather*}
\left(\begin{matrix}ae^{\chi_0}&B\\0
&(ae^{\chi_0})^{-1}\end{matrix}\right),
\end{gather*}
where
\begin{gather*}
B=a^{-1}e^{-\chi_0-2\chi_-}Y+ae^{\chi_0+2\chi_+}X^*.
\end{gather*}

We claim that the rank two bundle def\/ined by this transition function is semistable. Suppose otherwise.
Then there exists a sub-line bundle which has positive degree. The degree of this sub-line bundle is the negative of
the degree of a transition function $S\to {\rm GL}(1,\mathbb C)$ (see our conventions for transition functions
in Remark~\ref{transitionconvention}). This means that there exists a~factorization of the form
\begin{gather*}
\left(\begin{matrix}ae^{\chi_0}&B\\0
&(ae^{\chi_0})^{-1}\end{matrix}\right)=g_-\left(\begin{matrix}\lambda&\beta\\0&\lambda^{-1}\end{matrix}\right)g_+^{-1},
\end{gather*}
where $g_{\pm}=\left(\begin{matrix}\alpha_{\pm}&\beta_{\pm}\\\gamma_{\pm}&\delta_{\pm}\end{matrix}\right)$
are holomorphic ${\rm SL}(2,\mathbb C)$-valued functions in $\Sigma$ ($\Sigma^*$, respectively),
and the degree of $\lambda\colon S \to {\rm GL}(1,\mathbb C)$ is negative. This factorization implies
\begin{gather*}
ae^{\chi_0}\gamma_-=\lambda^{-1}\gamma_+
\end{gather*} and hence the line bundle def\/ined by the transition function
$(ae^{-\chi_0}\lambda)^{-1}$ has
a global holomorphic section. But this is impossible, because the degree of this transition function is positive (i.e., the degree
of the corresponding line bundle is negative).
This implies that~$E(g)$ is semistable. (Note that in general, a principal bundle is semistable if and only if the
corresponding adjoint bundle is semistable; however for ${\rm SL}(2,\mathbb C)$ it suf\/f\/ices to consider the def\/ining representation, see~\cite{Ramanathan}.)

Because $\chi_0=-\chi_0^*$, and because~$a$ is a positive constant, the line bundle def\/ined by the transition function $ae^{\chi_0}$
has antiholomorphic symmetry with respect to~$R$. This completes the proof.
\end{proof}

\subsection{A generalization of Theorem~\ref{SU(2)theorem2}}\label{theorem2generalization}

\begin{Theorem}\label{SU(2)theorem2generalized} Suppose $g\in C^{\infty}(S,{\rm SU}(2))$. If $g$ has a factorization
\begin{gather*}
g(z)=k_1^*(z)\left(\begin{matrix} e^{\chi(z)}&0\\
0&e^{-\chi(z)}\end{matrix}\right)k_2(z),
\end{gather*} where $\chi \in
C^{\infty}(S,i\mathbb R)$, and $k_1$ and $k_2$ are as in {\rm (I.1)} and~{\rm (II.1)}, respectively, of Theorem~{\rm \ref{SU(2)theorem1smooth2}},
then~$E(g)$, the holomorphic $G$ bundle on $\widehat{\Sigma}$
defined by~$g$ as a transition function, is semistable.
\end{Theorem}

\begin{proof} For $a_1$ and $a_2$ as in Theorem~\ref{SU(2)theorem1smooth2} (viewed as transition functions)
\begin{gather*}
\deg(a_1^*a_2)=\deg \big(a_1^{-1}\big)+\deg(a_2).
\end{gather*}
Since $a_1^{-1}$ and $a_2$ do not vanish in $\Sigma$, this degree is the negative of the sum of the number of common zeros of $(a,b)$ in $\Sigma$ and the number of common zeros of~$(c,d)$ in~$\Sigma$. In particular the degree of $a=a_1^*a_2$ is negative. With the exception of the last
paragraph, we can now repeat the preceding proof verbatim; the key point is that the degree of the transition
function $(ae^{-\chi_0}\lambda)^{-1}$, calculated in the penultimate paragraph, remains positive.
\end{proof}

The main point of this theorem is that the antiholomorphic symmetry has been broken. Consequently there is now a chance that the
factorization of $g$, as in the theorem, is generic.

\subsection{Do there exist converses?}\label{generic}

As of this writing, the questions of whether there are converses to Theorems~\ref{SU(2)theorem2} and~\ref{SU(2)theorem2generalized}
are open. We brief\/ly discuss the issues involved. In this subsection, for simplicity, we assume that~$S$ has one connected
component. We f\/irst consider Theorem~\ref{SU(2)theorem2}.

Suppose that $g\in C^{\omega}(S;{\rm SU}(2,\mathbb C))$ and $E(g)$ is semistable. Any holomorphic vector bundle on a Riemann surface
has a f\/lag of holomorphic subbundles. In particular for $E(g)$, there exists a holomorphic sub-line bundle in the
associated vector bundle for the def\/ining representation. This implies the existence of a factorization of the form
\begin{gather}\label{stablefactorization}
g=\left(\begin{matrix}l_{11}&l_{12}\\l_{21}&l_{22}\end{matrix}\right)
\left(\begin{matrix}\lambda&B'\\0&\lambda^{-1}\end{matrix}\right)
\left(\begin{matrix}u_{11}&u_{12}\\u_{21}&u_{22}\end{matrix}\right),
\end{gather}
where $l\in H^0(\Sigma^*,{\rm SL}(2,\mathbb C))$, $u\in H^0(\Sigma,{\rm SL}(2,\mathbb C))$,
$\lambda\colon S\to {\rm GL}(1,\mathbb C)$ has degree
zero (by semistability), and $B'\colon S\to \mathbb C$. To proceed rigorously, we would need to know more about
this kind of factorization, but in this heuristic discussion, we will put this aside. If we assume that the line bundle has an antiholomorphic
symmetry compatible with $R$,
then the degree of the line bundle is zero, and because we are assuming $S$ is connected, we can
assume that $\lambda=a'\exp(\chi_0)$, where $\chi_0=-\chi_0^*$ and $a'$ is a positive constant (we eventually want
$a'=a_1a_2$). By comparing with the factorization of $g$ in the previous subsection, we see that (putting questions about uniqueness aside)
\begin{gather*}
l_{11}=\alpha_1^*\exp(\chi_-), \qquad l_{21}=\beta_1^*\exp(\chi_-), \qquad u_{21}=\gamma_2\exp(-\chi_+),\qquad u_{22}=\delta_2\exp(-\chi_+).
\end{gather*}
Given this, one can possibly mimic the calculations
in~\cite[Section~3]{PJFA} to reconstruct the fac\-tori\-za\-tion. A~key point here is that we need to know there does exist a unitary
transition function; this is discussed in Appendix~\ref{intuition}.

Theorem \ref{SU(2)theorem2generalized} asserts the existence of a diagram of the form,
\begin{gather*}
\begin{matrix}\{g=k_1^*\operatorname{diag}(e^{\chi},e^{-\chi})k_2\}&\subset &C^{\omega}(S;{\rm SU}(2))& \subset &\operatorname{Hyp}(S;{\rm SL}(2))\\
\downarrow& & \downarrow & & \downarrow\\
\left\{\begin{matrix}\text{semistable}\\ \text{SL(2)-bundles}\end{matrix}\right\}& \subset & \left\{\begin{matrix}\text{holomorphic}\\
\text{SL (2)-bundles}\end{matrix}\right\}&\equiv &
\left\{\begin{matrix}\text{holomorphic}\\ \text{SL(2)-bundles}\end{matrix}\right\}\end{matrix}
\end{gather*}
where $k_1$ and $k_2$ satisfy the conditions in Theorem~\ref{SU(2)theorem1smooth2}, and isomorphic bundles are identif\/ied. The question is whether the f\/irst down arrow is onto (for the second down arrow, see Appendix~\ref{intuition}).
To understand what is needed, suppose that~$E(g)$ is semistable and there is a~factorization
as in~(\ref{stablefactorization}). To get started, we need a factorization
of the form $\lambda'=a_1^*a_2$ (for some~$\lambda'$ equivalent to~$\lambda$, as a~transition function), where~$a_1$ and~$a_2$ are functions as in Section~\ref{thm1refinement} and Theorem~\ref{SU(2)theorem2generalized}.

\begin{Remark}\label{positivetransfn}
To develop some intuition, we need to be able to solve the following problem:
Consider a decomposable bundle def\/ined by a transition function $\operatorname{diag}(a,a^{-1})$ where $a$ is positive function along~$S$ (e.g., $a=\exp(\chi_1)$, where~$\chi_1$ is a~zero mode with $\chi_1=\chi_1^*$). Find a multiloop $g\colon S\to {\rm SU}(2,\mathbb C)$ which maps to this bundle.
\end{Remark}

\section{Spin Toeplitz operators}\label{SU(2)thm3}

A spin structure for a Riemann surface is the same thing as a choice of square root for
the cano\-ni\-cal bundle, $\kappa$ (by way of explanation, holomorphic sections of $\kappa$ are holomorphic dif\/ferentials).
We let $\kappa^{1/2}$ ($\to \Sigma,\widehat{\Sigma},\Sigma^*$, depending on the surface) denote the choice of square root
(see  \cite[Chapter~7]{Segal} for other points of view). There is a natural (pre-) Hilbert space structure on
$\Omega^{0}\big(S,\kappa^{1/2}\big)\otimes \mathbb C^2$ ($\mathbb C^2$ valued spinors along~$S$) given by
\begin{gather*}
\langle\psi,\phi\rangle=\int_S\left( \psi,\overline{\phi} \right)_{\mathbb C^2},
 \end{gather*}
where $\left(v,\overline{w}\right)_{\mathbb C^2}$ denotes the standard Hermitian inner product on~$\mathbb C^2$, so that the integrand
is a~one density on~$S$.

For the sheaf of holomorphic sections of $\kappa^{1/2}$, and (a slight open thickening of) the covering of $\widehat{\Sigma}$ by
$\Sigma$ and $\Sigma^*$, there is a Mayer--Vietoris long exact sequence,
\begin{gather*}
0\to H^0\big(\widehat{\Sigma},\kappa^{1/2}\big) \to H^0\big(\Sigma,\kappa^{1/2}\big)\oplus H^0\big(\Sigma^*,\kappa^{1/2}\big)\\
\hphantom{0}{}
\to \Omega^{0}\big(S,\kappa^{1/2}\big) \to H^1\big(\widehat{\Sigma},\kappa^{1/2}\big)\to 0.
\end{gather*} In terms of the $\overline \partial$ operator,
$H^0\big(\widehat{\Sigma},\kappa^{1/2}\big)=\operatorname{ker}(\overline{\partial})$ and $H^1\big(\widehat{\Sigma},\kappa^{1/2}\big)=\operatorname{coker}(\overline{\partial})$.
The index of $\overline{\partial}$ is zero.

When $\overline{\partial}$ is invertible, a generic condition,
by taking Hilbert space completions, we obtain a~Hilbert space polarization
\begin{gather*}
L^2\Omega^{1/2}(S)\otimes \mathbb C^2=H_+\oplus H_-,
\end{gather*}
where $L^2\Omega^{1/2}(S)$, $H_+$, and $H_-$ denote the completions of $\Omega^{0}\big(S,\kappa^{1/2}\big)$, $H^0\big(\Sigma,\kappa^{1/2}\big)\otimes \mathbb C^2$, and $H^0\big(\Sigma^*,\kappa^{1/2}\big)\otimes \mathbb C^2$, respectively.
When $\overline{\partial}$ is not invertible, there are two distinct reasonable polarizations. The f\/irst possibility is that $H_+$ is the completion of $\Omega^{0}\big(S,\kappa^{1/2}\big)\otimes \mathbb C^2$, and $H_-$ is the orthogonal complement.
The second possibility is that $H_-$ is the completion of $H^0\big(\Sigma^*,\kappa^{1/2}\big)\otimes \mathbb C^2$ and~$H_+$ is the orthogonal complement.

Given an essentially bounded matrix valued function $g\colon S\to \mathcal L(\mathbb C^2)$,
there is an associated bounded multiplication operator
$M_g$ on $\Omega^{1/2}(S)\otimes \mathbb C^2$, and relative to a polarization as in the preceding paragraphs
\begin{gather*}
M_g=\left(\begin{matrix}A(g)&B(g)\\C(g)&D(g)\end{matrix}\right).
\end{gather*}
If $\Sigma=D$, the unit disk, then $A(g)$ and $B(g)$ are the classical (block) Toeplitz and Hankel operators associated to the symbol~$g$ (see~\cite{BS} and  \cite[Chapter~3]{Peller}). In general the basic qualitative properties of the operators~$A(g)$ and~$B(g)$ (which we
refer to as `spin Toeplitz and Hankel operators', respectively) are the same
as in the classical case, because the projection to $H_+$ dif\/fers from the classical projection by a smoothing operator.
For example (of most importance for our purposes), exactly as in~\cite[Chapter~5]{PS}, the map $g\to M_g$ def\/ines an embedding of loops into the Hilbert--Schmidt general linear
group; the precise statement is that there is an embedding
\begin{gather}\label{embedding}
L^{\infty}\cap W^{1/2}(S;{\rm SL}(2,\mathbb C)) \to {\rm GL}(H_+\oplus H_-)_{(2)}.
\end{gather}
The relevant Lie structure is described in~\cite[Chapters~6 and~7]{PS}.

\begin{Remark} The Krichever--Novikov theory, at least in principle, provides an explicit means to calculate spin Hankel and Toeplitz operators.
According to Krichever and Novikov (following the notation in~\cite{KN}), for certain spin structures and for generic basepoints, there is
an (orthonormal) basis for $\mathcal M^{1/2}$ (meromorphic spinors which are regular in the complement of the basepoints) of the form
\begin{gather*}\dots, \phi_{3/2},\phi_{1/2},\phi_{-1/2}\phi_{-3/2},\dots
\end{gather*}
(which in the classical case reduces to the usual basis
\begin{gather*}
\dots, z^2(dz)^{1/2},z(dz)^{1/2},z^0(dz)^{1/2},z^{-1}(dz)^{1/2},\dots)
\end{gather*}
a basis for $\mathcal M^0$, $\{f_k\colon k\in \frac{g}{2}+\mathbb Z\}$ (which in the classical case reduces to the usual basis),
and constants $Q$ such that
\begin{gather*}
f_n \phi_m=\sum_{\vert k\vert\le g/2} Q^k_{n,m} \phi_{n+m-k}.
\end{gather*}
Thus for each $n$, the matrices for the multiplication operator, and the spin Toeplitz operator, def\/ined by~$f_n$ will have band width (the number of nonzero diagonals that
appear) equal to $g+1$, $g=\operatorname{genus}(\hat{\Sigma})$ (in the classical case $f_n$ is a shift operator, and the bandwidth is one).
We will pursue this in part~II of this paper~\cite{BP2}.
\end{Remark}

\subsection{Determinants: proof of Theorem \ref{SU(2)theorem3}}

\begin{Theorem}\label{SU(2)theorem3statement} Suppose that $g$ has a factorization as in Theorem~{\rm \ref{SU(2)theorem2}}.
Then for any choice of spin structure and polarization as in the preceding subsection,
\begin{gather*}
\det\big(A(g)A\big(g^{-1}\big)\big)=\det\big(A(k_1)A\big(k_1^{-1}\big)\big)\det\big(A (e^{\chi})A(e^{-\chi})\big)\det\big(A(k_2)A\big(k_2^{-1}\big)\big).
\end{gather*}
\end{Theorem}

This theorem will follow from Theorem \ref{SU(2)theorem4b} and the lemmas below. To place the calculations in some context, recall
(see~\cite[pp.~88--89]{PS}) that for the universal central extension (for the Hilbert--Schmidt general
linear group)
\begin{gather*}0\to \mathbb C^* \to \widetilde{{\rm GL}}(H_+\oplus H_-)_{(2)}\to {\rm GL}(H_+\oplus H_-)_{(2)}\end{gather*}
and the local cross-section (def\/ined on the open dense set of $g$ such that $A(g)$ is invertible)
\begin{gather*}g\to [g,A(g)]\end{gather*}
the multiplication is given by
\begin{gather*}
\tilde g_1 \tilde g_2=\mathbf c(g_1,g_2)\tilde g_3,
\end{gather*}
where the
cocycle is given by
\begin{gather}
\mathbf c(g,h)^{-1}=\det\big(A(gh)A(h)^{-1}A(g)^{-1}\big)=\det\big(A(g)^{-1}A(gh)A(h)^{-1}\big)\nonumber\\
\hphantom{\mathbf c(g,h)^{-1}}{}
=\det\big(1+A(g)^{-1}B(g)C(h)A(h)^{-1}\big).\label{cocycle}
\end{gather}

The corresponding Lie algebra cocycle is given by
\begin{gather}
\omega\colon \ \mathfrak{gl}(H_+\oplus H_-)_{(2)} \times \mathfrak{gl}(H_+\oplus H_-)_{(2)} \to\mathbb C,
\nonumber\\
 \omega(X,Y)=\operatorname{tr}([A(X),A(Y)]-A([X,Y]))=\operatorname{tr}(B(Y)C(X)-B(X)C(Y)).
\label{liealgcocycle}
\end{gather}

We are interested in the restriction of the cocycle to $\Omega^0(S,{\rm SL}(2,\mathbb C))$ via the injection (\ref{embedding}).
The basic fact about this cocycle, which we will use repeatedly, is that if $g_-\in H^0(\Sigma^*,{\rm SL}(2,\mathbb C))$
and $g_+\in H^0(\Sigma,{\rm SL}(2,\mathbb C))$, then
\begin{gather}\label{equividentity}
\mathbf c(g_-g,hg_+)=\mathbf c(g,h).
\end{gather}
This follows from $A(g_-g)=A(g_-)A(g)$ and $A(hg_+)=A(h)A(g_+)$, which in turn follow from $B(g_-)=0$ and $C(g_+)=0$, respectively.
To see this, suppose that $H_+$ is the completion of $\Omega^{0}\big(\Sigma,\kappa^{1/2}\big)\otimes \mathbb C^2$. Given a spinor $\phi_+\in H_+$,
the product~$g_+\phi_+$ is holomorphic in~$\Sigma$, hence $g_+\phi_+\in H_+$. This implies~$C(g_+)=0$. Now consider a~spinor $\phi\in H_-$, i.e., $\phi$ is orthogonal to~$H_+$. We claim that~$g_-\phi$ is also orthogonal to~$H_+$. If $\psi_+\in H_+$,
then the pointwise inner product (along~$S$)
\begin{gather*}
\left(\psi,\overline{g_-\phi}\right)_{\mathbb C^2}=\left(g_-^*\psi_+,\overline{\phi}\right)_{\mathbb C^2}.
\end{gather*}
Since $g_-^*\in H^0(\Sigma,{\rm SL}(2,\mathbb C))$, it follows that $\langle \psi_+,g_-\phi\rangle=\langle g_-^*\psi_+,\phi\rangle=0$.
It follows that \mbox{$B(g_-)=0$}. A similar argument applies if~$H_+$ is the orthogonal complement of~$\Omega^{0}(\Sigma^*,\kappa^{1/2})\otimes \mathbb C^2$.

To prove the theorem we need to show that $\mathbf c(g,h)=1$ when $(g,h)=(k_1^*,k_2)$ and so on. In fact in some cases
we can prove much stronger statements.

\begin{Theorem}\label{SU(2)theorem4b}
$A(k_1^*k_2)=A(k_1^*)A(k_2)$. In particular $\mathbf c(k^*_1,k_2)=1$.
\end{Theorem}

In the process of proving the f\/irst statement, we will prove Theorem~\ref{SU(2)theorem4}, i.e., we will
only assume that~$k_1$ and~$k_2$ are measureable maps $S\to {\rm SU}(2)$ of the appropriate form. We need to assume
the loops are~$W^{1/2}$ for the second statement to be valid.

\begin{proof} Because
\begin{gather*}A(k_1^*k_2)=A(k_1^*)A(k_2)+B(k_1^*)C(k_2)
\end{gather*} this equivalent to showing that
$B(k_1^*)C(k_2)=0$. We will prove this by direct calculation.
Suppose that $f:=\left(\begin{matrix}f_1\\f_2\end{matrix}\right)\in H_+=H^{1/2}(\Delta)$.
Then
\begin{gather*}
B(k_1^*)C(k_2)f=P_+\left(k_1^*P_-(k_2f)\right)=P_+\left(k_1^*\left(\begin{matrix}P_-(d^*f_1-c^*f_2)\\0\end{matrix}\right) \right)\\
\hphantom{B(k_1^*)C(k_2)f}{}
=P_+\left(\left(\begin{matrix}a^*P_-(d^*f_1-c^*f_2)\\0\end{matrix}\right)\right)=0.
\end{gather*}
Thus $B(k_1^*)C(k_2)=0$. By relation~(\ref{cocycle}) (for suf\/f\/iciently smooth loops) this implies that \mbox{$\mathbf c(k^*_1,k_2)=1$}.
\end{proof}

\begin{Lemma}
$\mathbf c(k_1^*,\operatorname{diag}(\exp(\chi),\exp(-\chi)))=\mathbf c(\operatorname{diag}(\exp(\chi),\exp(-\chi)),k_2)=1$.
\end{Lemma}

\begin{proof} There exists a
generalized triangular factorization
\begin{gather*}\left(\begin{matrix}a^*&-b\\b^*&a\end{matrix}\right)=
\left(\begin{matrix}\alpha_1^*&\gamma_1^*\\\beta_1^*&\delta_1^*\end{matrix}\right)
\left(\begin{matrix}a_1&0\\0&a_1^{-1}\end{matrix}\right)\left(\begin{matrix}1&y+y_0\\0&1\end{matrix}\right).
\end{gather*}
Consequently, using (\ref{equividentity}),
\begin{align*}
\mathbf c\left(\left(\begin{matrix}a^*&-b\\b^*&a\end{matrix}\right),\left(\begin{matrix}\exp(\chi)&0\\0&\exp(-\chi)\end{matrix}\right)\right)
& =
\mathbf c\left(\left(\begin{matrix}1&y+y_0\\0&1\end{matrix}\right),\left(\begin{matrix}e^{\chi}&0\\0&e^{-\chi}\end{matrix}\right)\right)\\
& =\mathbf c\left(\exp\left(\left(\begin{matrix}0&y+y_0\\0&0\end{matrix}\right) \right),\exp\left(\left(\begin{matrix}\chi&0\\0&-\chi\end{matrix}\right) \right)\right).
\end{align*}

In the last line we are considering the cocycle pairing of an element of $\Omega^0(S,N^+)$, where $N^+$ is the group of unipotent
upper triangular matrices in ${\rm SL}(2,\mathbb C)$, and an element of~$\Omega^0(S,H)$,
where~$H$ is the diagonal subgroup of~${\rm SL}(2,\mathbb C)$.
We claim that the corresponding Lie algebra cocycle pairing for $\Omega^0(S,\mathfrak n^+)$ and $\Omega^0(S,\mathfrak h)$ is zero.
To check this
suppose that $X:=\left(\begin{matrix}0&\beta\\0&0\end{matrix}\right)\in \Omega^0(S,\mathfrak n^+)$ and
$Y:=\operatorname{diag}(\chi,-\chi)\in \Omega^0(S,\mathfrak h)$.
We must calculate~(\ref{liealgcocycle}) for these two multiloops. Then
\begin{gather*}
B(X)C(Y)\left(\left(\begin{matrix}f_1\\f_2\end{matrix}\right)\right)=\left(\begin{matrix}-(Y(\chi f_2)_-)_+\\0\end{matrix}\right).
\end{gather*}
From this it is clear that $\operatorname{tr}(B(X)C(Y))=0$. Similarly $\operatorname{tr}(B(Y)C(X))=0$. Thus the cocycle pairing for
$\Omega^0(S,\mathfrak n^+)$ and $\Omega^0(S,\mathfrak h)$ is zero. It follows that the cocycle pairing for
$\Omega^0(S,N^+)$ and $\Omega^0(S,H)_0$ (the identity component) is trivial.
This implies that
\begin{gather*}
\mathbf c(k_1^*,\operatorname{diag}(\exp(\chi),\exp(-\chi)))=1.
\end{gather*}
By the same argument $\mathbf c(\operatorname{diag}(\exp(\chi),\exp(-\chi)),k_2)=1$.
\end{proof}

\begin{Lemma} $\mathbf c(k_1^*\operatorname{diag}(\exp(\chi),\exp(-\chi)),k_2)=\mathbf c(k_1^*,\operatorname{diag}(\exp(\chi),\exp(-\chi))k_2)=1$.
\end{Lemma}

\begin{proof}Using (\ref{equividentity})
\begin{gather*}\mathbf c(k_1^*\operatorname{diag}(\exp(\chi),\exp(-\chi)),k_2)=\mathbf c\left(\left(\begin{matrix}e^{\chi}&y+y_0\\0&e^{-\chi}\end{matrix}\right),
\left(\begin{matrix}1&x^*+x_0\\0&1\end{matrix}\right)\right).
\end{gather*}
As in the proof of the preceding lemma, this is a~cocycle pairing between elements from two (connected) groups. The corresponding
pairing of Lie algebra elements is trivial. So this cocycle pairing is trivial.
\end{proof}

These lemmas imply Theorem~\ref{SU(2)theorem3statement}.

\subsection{Calculating the determinant: scalar case}\label{scalarcase}

Our goal is to calculate the second determinant appearing on the right-hand side of the statement
in Theorem~\ref{SU(2)theorem3statement}. Recall that in that statement we think of~$e^{\chi}$
as a multiplication operator on scalar valued functions. We f\/irst consider the simplest possibility,
which can be handled in the same way as done by Widom in the classical case (see~\cite{W}).

\begin{Proposition} Suppose that $\chi=\chi_++\chi_-$, where $\chi_+$ is holomorphic in $\Sigma$
and $\chi_-$ is holomorphic in $\Sigma^*$. Then
\begin{enumerate}\itemsep=0pt
\item[$(a)$] $\det(A(e^{\chi})A(e^{-\chi}))=\exp(\operatorname{tr}(B(\chi_+)C(\chi_-)))$,
\item[$(b)$] if $\chi$ is $i\mathbb R$ valued, then this equals
\begin{gather*}\exp(-\operatorname{tr}(B(\chi_+)B(\chi_+)^*))=\exp\big({-}\vert B(\chi_+)\vert^2_{\mathcal L_2}\big).
\end{gather*}
\end{enumerate}
In the classical case $\Sigma=D$,
\begin{gather*}\operatorname{tr}(B(\chi_+)B(\chi_+)^*)=\sum_{n=1}^{\infty}n\vert \chi_n\vert^2.
\end{gather*}
\end{Proposition}

\begin{proof}
Because $A(e^{\chi})=A(e^{\chi_-})A(e^{\chi_+})$, $\det(A(e^{\chi})A(e^{-\chi}))$ equals
\begin{gather*}
\det\big( A(e^{\chi_-})A(e^{\chi_+})A(e^{-\chi_-})A(e^{-\chi_+})\big)=\det\big(e^{A(\chi_-)}e^{A(\chi_+)}e^{A(-\chi_-)}e^{A(-\chi_+)}\big).
\end{gather*}
The Helton--Howe formula (see \cite{HH}) implies that this equals $\exp(\operatorname{tr}([A(\chi_-),A(\chi_+)]))$.
Now observe that{\samepage
\begin{align*}
\begin{split}[A(\chi_-),A(\chi_+)]& =A(\chi_-)A(\chi_+)-A(\chi_+)A(\chi_-)\\
& =A(\chi_-)A(\chi_+)-\left(A(\chi_+\chi_-)-B(\chi_+)C(\chi_-)\right)=B(\chi_+)C(\chi_-).
\end{split}
\end{align*}
This implies part~(a).}

For part (b) the only thing we need to comment on is the last equality. This follows from writing out the matrix
for~$B$, relative to the standard basis. Up to a multiple, it also follows abstractly by the~${\rm SU}(1,1)$ symmetry
of~$D$ and Schur's lemma.
\end{proof}

\begin{Remark}
It is obvious that the formula in~(a) depends on the complex structure of~$\Sigma$. Does it depend on the spin structure?
\end{Remark}

The shortcoming of the preceding proposition is that when $\text{genus}(\widehat{\Sigma})>0$, we have not accounted for
zero modes.

\begin{Proposition}
\begin{gather*}
\det\big(A(e^{\chi})A(e^{-\chi})\big)=\mathbf c(e^{\chi_-},e^{\chi_0})\mathbf c(e^{\chi_0},e^{\chi_+})
\det\big(\big\{e^{-A(\chi_+)},e^{A(\chi_-)}\big\}\big)\det\big(A(e^{\chi_0})A(e^{-\chi_0})\big).
\end{gather*}
In the case that $\chi$ has values in $i\mathbb R$, this equals
\begin{gather*}
\det\big(\big\vert \big\{A(e^{-\chi_0}),e^{-A(\chi_+)}\big\}\big\vert^2\big)\det\big(A(e^{\chi_0})A(e^{\chi_0})^*\big)
\exp\big({-}\operatorname{tr}(B(\chi_+)B(\chi_+)^*)\big),\end{gather*}
where $\{A,B\}=ABA^{-1}B^{-1}$ denotes the group commutator.
\end{Proposition}

\begin{proof}
\begin{gather*}
A(e^{\chi})A(e^{-\chi})=e^{A(\chi_-)}A(e^{\chi_0})e^{A(\chi_+)}e^{-A(\chi_-)}
A(e^{-\chi_0})e^{-A(\chi_+)}.
\end{gather*}
The right-hand side is of the form $xyzx^{-1}Yz^{-1}$. The conjugate of this by $z^{-1}$
equals
\begin{gather*}
z^{-1}xyzx^{-1}Y=\big\{z^{-1}x,y\big\}y\big\{z^{-1},x\big\}y^{-1}yY,
\end{gather*}
where we have temporarily assumed that $y$ is invertible.
This implies
\begin{gather*}
\det\big(A(e^{\chi})A(e^{-\chi})\big) =\det\big(\big\{e^{-A(\chi_+)} e^{A(\chi_-)},A(e^{\chi_0})\big\}\big)\\
\hphantom{\det\big(A(e^{\chi})A(e^{-\chi})\big) =}{}\times
\det\big(\big\{e^{-A(\chi_+)},e^{A(\chi_-)}\big\}\big)\det\big(A(e^{\chi_0})A(e^{-\chi_0})\big)
\end{gather*}
(each of these determinants is the determinant of an operator of the form identity + trace class).
Using the Helton--Howe formula, this equals
\begin{gather*}
\det\big(\big\{e^{-A(\chi_+)} e^{A(\chi_-)},A(e^{\chi_0})\big\}\big)
\exp\big({-}\operatorname{tr}(B(\chi_+)C(\chi_-))\big)\det\big(A(e^{\chi_0})A(e^{-\chi_0})\big).
\end{gather*}

Now we use the identity
\begin{gather*}
\big\{z^{-1}x,y\big\}=z^{-1}xyx^{-1}y^{-1}zz^{-1}yzy^{-1}.
\end{gather*}
This implies that the f\/irst determinant in the
expression above satisf\/ies
\begin{gather*}
\det\big(\big\{e^{-A(\chi_+)} e^{A(\chi_-)},A(e^{\chi_0})\big\}\big)=
\det\big(\big\{e^{A(\chi_-)},A(e^{\chi_0})\big\}\big)\det\big(\big\{e^{A(\chi_+)},A(e^{\chi_0})\big\}\big).
\end{gather*}
In terms of the group cocycle notation, this equals
\begin{gather*}
\mathbf c(e^{\chi_-},e^{\chi_0})\mathbf c(e^{\chi_0},e^{\chi_+}).
\end{gather*}
This completes the proof.
\end{proof}

\begin{Remark}In the classical case this determinant is nonvanishing. It would seem unlikely
that this is true in general. But the expression we have produced does not seem to help in
deciding this question. It should be possible to calculate the determinant for~$\exp(\chi_0)$ exactly.
\end{Remark}

\appendix
\section{Hyperfunctions and holomorphic bundles}\label{hyperfunctions}

\subsection{Hyperfunctions on a circle}

Suppose that $X$ is an oriented analytic compact $d$-manifold. Let
$C^{\omega}\Omega^{d}(X)$ denote the vector space of real analytic forms of order
$d$ on $X$. There is a standard topology on this space.
By def\/inition the space of (real) hyperfunctions on $X$, denoted $\operatorname{Hyp}(X;\mathbb R)$,
is the dual of this topological vector space; the
space of complex hyperfunctions is the complexif\/ication, denoted $\operatorname{Hyp}(X;\mathbb C)$, or more simply
$\operatorname{Hyp}(X)$ (see  \cite[p.~35]{Helgason}).
In the case of $X=S^1$, any complex hyperfunction has a unique representation of the
form
\begin{gather*}
f=f_-+f_0+f_+,
\end{gather*}
where $f_-\in H^0(\Delta^*,(\infty);\mathbb C,0)$, $f_0\in \mathbb C$, and $f_+\in H^0(\Delta,0;\mathbb C,0)$
(this is how hyperfunctions were f\/irst introduced, see~\cite{Sato}).
The corresponding functional is given by
\begin{gather*}
C^{\omega}\Omega^{1}\big(S^1\big)\to \mathbb C\colon \ \omega\to \int_{\vert z\vert
=1-\epsilon}(f_0+f_+)\omega+\int_{\vert z\vert=1+\epsilon}f_-\omega,
\end{gather*}
where $\omega=gdz$ and $g$ is complex analytic in a collar of
$S^1$, for suf\/f\/iciently small $\epsilon$.

An integrable function $f\colon S^1\to \mathbb C$ naturally def\/ines a hyperfunction by integrating an analytic form against $f$
around $S^1$. Note that for the linear triangular decomposition for the Fourier series of~$f$, where $f_{\pm}=\sum\limits_{\pm n>0}f_nz^n$,
$f_{\pm}$ are not
necessarily integrable on the circle, but they do def\/ine holomorphic functions in~$\Delta$ ($\Delta^*$, respectively).

\subsection{Nonabelian hyperfunctions on a circle}\label{hyperfunctiondfn}

Suppose that $G$ is a simply connected complex Lie group, e.g., $G={\rm SL}(2,\mathbb C)$.

In~\cite{PS} it is observed that the group of analytic loops,
$C^{\omega}(S^1;G)$, is a complex Lie group. A neighborhood
of the identity consists of those loops which have a
unique Riemann--Hilbert factorization
\begin{gather}\label{RHfactorization}
g=g_{-}\cdot g_0\cdot g_{+},
\end{gather}
where $g_{-}\in H^0(D^{*},(\infty) ;G,1)$, $g_0\in G$, $g_{+}\in H^
0(D,0;G,1)$. A~model for this neighborhood is
\begin{gather*}
H^1(D^{*},\mathfrak g)\times G\times H^1(D,\mathfrak g),
\end{gather*}
where the linear coordinates are determined by
$\theta_{+}=g_{+}^{-1}(\partial g_{+})$, $\theta_{-}=(\partial g_{
-})g_{-}^{-1}$. The (left or right) translates
of this neighborhood by elements of $L_{\rm f\/in}K$ cover
$H^0(S^1,G)$; a key point is that the transition functions
are functions of a~f\/inite number of variables, in an
appropriate sense (see~\cite[Chapter~2, Part~III]{Pi}).

The hyperfunction completion, $\operatorname{Hyp}(S^1,G)$, is modeled on the
space
\begin{gather*}
H^1(\Delta^{*},\mathfrak g)\times G\times H^1(\Delta ,\mathfrak g)
\end{gather*}
and the transition functions are
obtained by continuously extending the transition
functions for the analytic loop space of the preceding
paragraph. The global def\/inition is
\begin{gather}\label{hyperfunctiondnf}
\operatorname{Hyp}\big(S^1,G\big)=\lim_{r\downarrow 1}H^0\big(\{1<\vert z\vert <r\},G\big)\times_{
H^0(S^1,G)}\lim_{r\uparrow 1}H^0\big(\{r<\vert z\vert <1\},G\big).
\end{gather}
From this point of view, a hyperfunction is an equivalence class $[g,h]$, where $g$ ($h$) is a $G$-valued holomorphic
map in an annulus to the left (right, respectively) of~$S^1$). The correct way to topologize this space does
not seem clear from this point of view, and we will put this aside.
From the global def\/inition it is clear that the group
$H^0(S^1,G)$ acts naturally from both the left and right of
$\operatorname{Hyp}(S^1,G)$, e.g., for the left action, $g_l\colon [g,h]\to [g_lg,h]$. There is a generalized Birkhof\/f decomposition
\begin{gather*}
\operatorname{Hyp}\big(S^1,G\big)=\bigsqcup_{\lambda\in \operatorname{Hom}(S^1,T)}\Sigma^{\rm hyp}_{\lambda},\qquad\Sigma_{
\lambda}^{\rm hyp}=H^0 (\Delta^{*},G )\cdot\lambda\cdot H^0(\Delta ,G).
\end{gather*}
The top stratum (the
piece with $\lambda =1$ above) is open and dense, and for each
point~$g$ in the top stratum, there is a unique factorization as
in~(\ref{RHfactorization}), where $g_{\pm}$ are
$G$-valued holomorphic functions in the open disks~$\Delta$
and~$\Delta^{*}$, respectively. We will refer to
$g_{-},g_0,g_{+}$ ($\theta_{ -},g_0,\theta_{+}$, respectively) as
the Riemann--Hilbert coordinates (linear Riemann--Hilbert
coordinates, respectively) of~$g$ (see~\cite[Chapter~2,  Part~III]{Pi}).

Given a continuous function $g\colon S^1 \to G$ there exists a generalized Riemann--Hilbert factorization,
where the factors $g_{\pm}$ are not necessarily continuous (see  \cite[Theorem 1.1 of Chapter~VIII]{CG}
for a precise statement). In this way
any reasonable loop in~$G$ can be regarded as an element of $\operatorname{Hyp}(S^1,G)$.

The space $\operatorname{Hyp}(S^1,G)$ depends only on the orientation and real analytic structure of~$S^1$, i.e., there
is a natural action of real analytic homeomorphisms on hyperfunctions: in terms of the global def\/inition~(\ref{hyperfunctiondnf})
\begin{gather*}
\sigma\colon \ [g,h]\to \big[\sigma^{-*}g,\sigma^{-*}h\big].
\end{gather*}
Consequently if $S$ is an oriented real analytic one-manifold, then $\operatorname{Hyp}(S,G)$ is well-def\/ined.

\subsection[Hyperfunctions and holomorphic $G$ bundles]{Hyperfunctions and holomorphic $\boldsymbol{G}$ bundles}\label{bundles}

Suppose that $\widehat{\Sigma}$ is a closed Riemann surface. In this subsection it is not necessary
to assume that $\widehat{\Sigma}$ is a~double. We also suppose that $c\colon S^1\to\widehat{\Sigma}$
is a~real analytic embedding, i.e., a~simple analytic loop. We could more generally suppose that $c$ is an embedding of several
disjoint analytic loops, but we will focus on one to simplify notation. A~basic (nongeneric) example is
the case when $\widehat{\Sigma}$ is a double and~$c$ is a parameterization of~$S$, the f\/ixed point set
of the involution~$R$.

Let $\mathcal O$ denote the structure sheaf of $\widehat{\Sigma}$. There is an
associated mapping
\begin{gather*}
E=E_{c,\mathcal O}\colon \ \operatorname{Hyp}\big(S^1,G\big)\to H^1(\mathcal O_G)\colon \ [g,h]\to E_{c}([g,h]),
\end{gather*}
where $H^1(\mathcal O_G)$ denotes the set of isomorphism classes of holomorphic
$G$-bundles on $\widehat{\Sigma}$. For an ordinary analytic $G$-valued loop $g$, $E$ maps the loop
to the isomorphism class of the holomorphic $G$-bundle def\/ined by using the loop as a~holomorphic transition function in a tubular neighborhood of the image of~$c$.

\begin{Remark}\label{transitionconvention} Our convention for transition functions is the following.
A~(holomorphic) section of~$E(g)$ (which one should think of as a frame) is a (holomorphic) function $s\colon \widehat{\Sigma}{\setminus} c\to G$ such that $s_+g^{-1}=s_-$, where $s_+ (s_-)$
is the restriction of $s$ to a suf\/f\/iciently small annulus to the left (respectively, the right) of the
oriented loop $c$. Consequently for an associated bundle $E(g)\times_G V$ a (holomorphic) section is represented by
a (holomorphic) function $v\colon \widehat{\Sigma}{\setminus} c\to V$ such that $gv_+=v_-$.

With this convention, for scalar transition functions, the degree of a line bundle is the negative of the degree of a corresponding
transition function, e.g., for dif\/ferentials on~$\mathbb P^1$, $dw=dz(-z^{-2})$ ($s_-=(dw)$ and $s_+=(dz)$); the degree of the canonical bundle is~$2$, and the degree of the transition function~$-z^2$ is~$-2$.
\end{Remark}

The point is that $E$ can be extended naturally to hyperfunctions in the following way. The map $c$ extends
uniquely to a holomorphic embedding
$c\colon \{1-\epsilon <\vert z\vert <1+\epsilon \}\to\widehat{\Sigma}$ for some $
\epsilon >0$. Given the pair
$(g,h)$, we obtain a holomorphic bundle on $\widehat{\Sigma}$ by using~$g$ as
a transition function on an $\epsilon'$-collar to the left of~$c$ and~$h$ as a transition function on an $\epsilon'$-collar to the right of~$c$, for some $\epsilon'<\epsilon$, depending upon the pair $(g,h
)$. The isomorphism class of this bundle is independent of the
choice of $\epsilon'$, and depends only upon $[g,h]\in \operatorname{Hyp}(S^1,G)$.

The basic properties of the mapping~$E_{c}$ are
summarized as follows.

\begin{Proposition}\label{transitionlemma} \quad
\begin{enumerate}\itemsep=0pt
\item[$(a)$] If $\phi\in C^{\omega}\operatorname{Hom}(S^1)$, then the induced map
\begin{gather*}
\operatorname{Hyp}\big(S^1,G\big)\stackrel{\phi}{\rightarrow} \operatorname{Hyp}\big(S^1,G\big)\stackrel{E_{c}}{\rightarrow}H^1(\mathcal O_G)
\end{gather*}
equals $E_{c\circ\phi^{-1}}$.
\item[$(b)$] There is a holomorphic action
\begin{gather*}
H^0\big(\widehat{\Sigma}{\setminus} c\big(S^1\big)\big)\times \operatorname{Hyp}\big(S^1,G\big)\to \operatorname{Hyp}\big(S^1,G\big)\colon \ f,[g,h]\to
\big[f\vert_{S^1_{-}}g,hf\vert_{S^1_{+}}^{-1}\big].
\end{gather*}

\item[$(c)$] Inclusion and the mapping $E_{c}$ induce isomorphisms of sets
\begin{gather*}
H^0\big(S^1,G\big)/H^0\big(\overline{\widehat{\Sigma}{\setminus} c},G\big)\to \operatorname{Hyp}\big(S^1,G\big)/H^0\big(\widehat{\Sigma}{\setminus} c\big)\to H^1(\mathcal O_G),
\end{gather*}
where $\overline{\widehat{\Sigma}{\setminus} c}$ denotes the closed Riemann surface obtained by cutting along $c$ and adding two boundary components
 $($see  {\rm \cite[Section~8.11]{PS})}.
 \end{enumerate}
\end{Proposition}

An alternate way to think about the projection $E_{c,\mathcal O}$ is
as follows. Suppose that $(g,h)$ represents $[g,h]\in \operatorname{Hyp}(S^1,G)$.
In the $C^{\infty}$ category the principal bundle def\/ined using the transition
functions~$(g,h)$ is trivial. Therefore we can f\/ind smooth functions
\begin{gather*}s\colon \  \widehat{\Sigma}{\setminus} c\big(S^1\big)\to G \qquad \text{and} \qquad s_0\colon \ S^{1\epsilon}\to G
\end{gather*}
(where $S^{1\epsilon}$ is an $\epsilon$-collar neighborhood of $S^
1$ in $\widehat{\Sigma}$, and
the $\epsilon$-collar will depend upon $(g,h)$) such that
\begin{gather*}
s_{-}=gs_0 \qquad \text{and} \qquad s_0=hs_{+}
\end{gather*}
on an $\epsilon$ collar to the left (respectively, the right) of~$
S^1$.
If $s$ is replaced by $s'$, then there exists $f\in C^{\infty}
(\widehat{\Sigma} ,G)$
such that $s'=sf$, i.e., $s_{-}'=s_{-}f$ and so on with~$+$ and~$
0$
in place of~$-$.

Def\/ine $a=s^{-1}\overline{\partial }s$. Then the gauge equivalence class of $
a$
depends only upon $[g,h]$. For if we change $g$, $h$ to
$g_{-}gg_0$, $g_0^{-1}hg_{+}$, where $g_{\pm}$ is holomorphic in $\widehat{\Sigma}^{
\pm}$, and $g_0$ is
holomorphic in~$S^{1\epsilon}$, then the choice of $s$ is modif\/ied in
an obvious way, and this does not change~$a$, and when
$s\to s'=sf$, then $a\to f^{-1}af+f^{-1}\overline{\partial }f$ (a gauge transformation).
Thus we obtain a well-def\/ined map
\begin{gather*}
\operatorname{Hyp}\big(S^1,G\big)\to\Omega^{0,1}\big(\widehat{\Sigma} ,\mathfrak g\big)/C^{\infty}(\widehat{\Sigma}, G)=H^1(\mathcal O_G).
\end{gather*}

\subsection[The image of $LK$]{The image of $\boldsymbol{LK}$}\label{intuition}

Given a simple loop $c$ on a closed Riemann surface $\widehat{\Sigma}$, $H^0(S^1,G)$ maps onto
the set of all isomorphism classes of holomorphic~$G$ bundles on~$\widehat{\Sigma}$. The main point of this subsection
is to show that, assuming $K$ is simply connected, the subgroup of~$K$ valued loops also maps onto the set of all isomorphism classes of holomorphic bundles.

We f\/irst explain why this is true generically, but with critical exceptions, for non-simply connected~$K$.

\begin{Proposition} Suppose $K=\mathbb T$. The map from circle valued loops $($of degree zero$)$
to isomorphism classes of holomorphic line bundles $($of degree zero$)$,
\begin{gather}\label{abeliancase}
E_c\colon \ (L\mathbb T)_0 \to \operatorname{Jac}\big(\widehat{\Sigma}\big)
\end{gather}
is not surjective if and only if~$c$ is a straight line with respect to the flat geometry defined by
some holomorphic differential.
\end{Proposition}

\begin{proof} We must determine when the connecting map, as in~(\ref{connectmap}),
\begin{gather*}
C^{\omega}(\operatorname{Im}(c);\mathbb R) \to H^{0,1}\big(\hat \Sigma\big)\colon \ f\to \bar \partial F
\end{gather*}
is surjective, where in the $C^{\infty}$ category, $f=F_+-F_-$, $F$ is a smooth function on the complement of~$c$,
with limits $F_{\pm}$ from the left (right, respectively) of~$c$. This map fails to be surjective if and only if
there exists a nonzero holomorphic dif\/ferential $\omega$ such that $\operatorname{Re}(\int_{\hat \Sigma} \bar\partial F\wedge \omega)=0$ for all~$f$, i.e.,
there is a nonzero real harmonic one form $\eta=\omega+\overline \omega$ such that $\int_c f\eta=0$ for all
real functions on~$c$, which is equivalent to saying that the pullback of~$\eta$ to~$c$ vanishes. Locally $\omega=wdz$ and the corresponding f\/lat metric (with conical singularities at the zeros of $\omega$) is $|dW|^2=|w|^2|dz|^2$. The vanishing of the pullback of
$\operatorname{Re}(dW)$ along $c$ is equivalent to saying that~$c$ is a straight line.
\end{proof}

Suppose that $\widehat{\Sigma}$
is a double of $\Sigma$, a compact Riemann surface with boundary, as in the text,
and $\operatorname{Im}(c)=S$. The Jacobian of $\widehat{\Sigma}$ f\/its into a short exact sequence
\begin{gather*}
0\to H^1\big(\widehat{\Sigma},2\pi \mathbb Z\big) \to H^1\big(\widehat{\Sigma},\mathbb R\big) \to \operatorname{Jac}\big(\widehat{\Sigma}\big)\to 0.
\end{gather*}
$R$ acts equivariantly on this sequence, using pullback. $R$ also acts equivariantly on the projection
$(LT)_0 \to \operatorname{Jac}(\widehat{\Sigma})$, where the action on transition functions along~$S$ is given by
\begin{gather*}
g \to g^*:=(R^*g)^{-1}.
\end{gather*} In this case the image consists of line bundles having an antiholomorphic ref\/lection symmetry
compatible with $R$, and these correspond to cohomology classes
which are f\/ixed by~$R$. This is precisely the kind of (sub-line bundle) degeneracy that we encountered in Theorem~\ref{SU(2)theorem2}.

Suppose now that $K$ is simply connected.
Given $R$ and the involution $\tau$ def\/ining $K$ as a~real form of~$G$, there is an involution of $\Omega^0(S,G)$
which f\/ixes $\Omega^0(S,K)\colon g\to \tau(g\circ R)$. This does not have an extension to $\operatorname{Hyp}(S,G)$ (in sharp contrast
to the abelian
case). For in the nonabelian case this involution is not compatible with the action in part~(b)
of Proposition~\ref{transitionlemma}. Consequently there does not exist a real form for isomorphism classes of $G$ bundles
on $\widehat{\Sigma}$ which would conf\/ine the image of $\Omega^0(S,K)$, as in the abelian case.

In the nonabelian case we conjecture that the ordered product map
\begin{gather*}H^0(\Sigma^*,G) \times \Omega^0(S,K)\times H^0(\Sigma,G) \to \Omega^0(S,G)\end{gather*}
is surjective. Our objective is to explain why this is plausible.

For a group $H$ use left translation to identify $H\times \mathfrak h$ with $TH$. With this convention
at the point $(g_-,k,g_+)$, the derivative of the ordered product mapping is given by the formula
\begin{gather*}
H^0(\Sigma^*, \mathfrak g) \times \Omega^0(S,\mathfrak k) \times H^0(\Sigma,\mathfrak g)\to \Omega^0(S,\mathfrak g),\\
 (X_-,x,X_+)\to \frac{d}{dt}\Big\vert_{t=0}\big(g_+^{-1}k^{-1}g_-^{-1}g_-e^{tX_-}ke^{tx}g_+e^{tX_+}\big)=
g_+^{-1}k^{-1}X_-kg_++g_+^{-1}xg_++X_+.
\end{gather*}

\begin{Lemma} The derivative of the ordered product map at $(g_-,k,g_+)$ is critical $($i.e., not surjective$)$ if
$k$ is degenerate in the sense that
\begin{gather*}
\cap_{z\in S^1}\operatorname{ker}(\operatorname{Ad}(k(z))\ne \{0\}.
\end{gather*}
\end{Lemma}

\begin{Remark} \quad
\begin{enumerate}\itemsep=0pt
\item[(a)] Conjecturally the converse to the lemma holds.

\item[(b)] In a heuristic way, this is telling us when the map
\begin{gather*}\Omega^0(S,K) \to H^0(\Sigma^*,G)\backslash\Omega^0(S,G)/H^0(\Sigma,G) \end{gather*}
is regular, i.e., has a surjective derivative.
\end{enumerate}
\end{Remark}

\begin{proof} The derivative of the ordered product map at $(g_-,k,g_+)$ is surjective if and only if the
the map
\begin{gather*}
H^0(\Sigma^*, \mathfrak g) \times \Omega^0(S,\mathfrak k) \times H^0(\Sigma,\mathfrak g)\to \Omega^0(S,\mathfrak g),\\
(X_-,x,X_+)\to X_-+x+kX_+k^{-1}
\end{gather*}
is surjective if and only if the dual map is one to one. To compute the dual map, we identify the dual of
$\Omega^0(S,\mathfrak g)$ with itself using the invariant form. The dual map is then given by
\begin{gather*}
\Omega^0(S,\mathfrak g)\to H^0(\Sigma^*, \mathfrak g) \times \Omega^0(S,\mathfrak k) \times H^0(\Sigma,\mathfrak g),\\
 \phi \to \big(\phi_-, \phi_{L\mathfrak k}, \big(k^{-1} \phi k\big)_+\big).
 \end{gather*}
This map is one to one if and only if the map (with f\/inite-dimensional domain)
\begin{gather}\label{degcond}
\{\phi_0\colon \phi_0^*=\phi_0\}\to H^0(\Sigma,\mathfrak g)\colon \phi_0 \to \big(k^{-1}(\phi_0+\phi_0^*)k\big)_+
\end{gather}
is one to one (here~$\{\phi_0\}$ is the vector space of zero modes with values in~$\mathfrak g$).
If~$k$ is degenerate as in the statement of the lemma, then given $X\in \operatorname{ker}(\operatorname{Ad}(k(z))$ for all~$z\in S^1$, and any zero
scalar zero mode~$\chi_0$ satisfying $\chi_0=\chi_0^*$, $\phi_0=\chi_0 X$ is in the kernel of~(\ref{degcond}).

Conversely suppose $\phi_0$ is in the kernel of~(\ref{degcond}). Then
\begin{gather*}
\psi:=k^{-1}\phi_0k=\psi_0+\psi_-=\psi^*=\psi_0^*+\psi_+.
\end{gather*}
Consequently $k^{-1}\phi_0k$ is another zero mode satisfying $\psi_0=\psi_0^*$.
This would seem to force~$k$ to be degenerate, but this is not entirely clear.
\end{proof}

At the point $(1,1,1)$, the derivative is the sum mapping
\begin{gather*}
H^0(\Sigma^*, \mathfrak g) \times \Omega^0(S,\mathfrak k) \times H^0(\Sigma,\mathfrak g)\to \Omega^0(S,\mathfrak g).
\end{gather*}
The image of this sum is proper and equal to
\begin{gather*}
\Omega^0(S,\mathfrak k) + H^0(\Sigma,\mathfrak g).
\end{gather*}
However at a generic point, $(g_-,k,g_+)$, because of noncommutativity, the derivative is surjective. At these points the map is locally open.

This derivative calculation strongly supports the conjecture, but
it is not clear how to parlay this into a~proof.

\pdfbookmark[1]{References}{ref}
\LastPageEnding

\end{document}